\documentclass{amsart}

\usepackage{geometry}
\usepackage{macros}
\hypersetup{
    pdftitle = {Cartier-Dieudonné theory of connected group schemes},
    pdfauthor = {Casimir Kothari, Joshua Mundinger},
}

\usepackage{enumitem}
\setlist[enumerate]{label = \roman*.}

\newcommand{\uM}{\underline{M}}

\DeclareMathOperator{\aPolyd}{\alepho-Polyd}
\DeclareMathOperator{\aFLG}{\alepho-FLG}
\DeclareMathOperator{\aCFLG}{\aleph-CFLG}
\DeclareMathOperator{\SpInd}{SpInd}
\DeclareMathOperator{\Prim}{Prim}

\newcommand{\bR}{{\overline{R}}}
\newcommand{\Ghat}{\widehat{\mathbb{G}}}

\theoremstyle{plain}
\newtheorem{theorem}{Theorem}[subsection]
\newtheorem{lemma}[theorem]{Lemma}
\newtheorem{proposition}[theorem]{Proposition}
\newtheorem{corollary}[theorem]{Corollary}

\newtheorem{theoremalpha}{Theorem}

\theoremstyle{definition}
\newtheorem{definition}[theorem]{Definition}
\newtheorem{example}[theorem]{Example}

\newtheorem{remark}[theorem]{Remark}

\usepackage[
  style=alphabetic,
  citestyle=alphabetic,
  maxnames=100, 
  doi=false,
  isbn = false,
  giveninits = true, 
  datamodel = mrnumber, 
]{biblatex}
\DeclareNameAlias{default}{family-given}

\DeclareFieldFormat{mrnumber}{%
  MR\addcolon\space
    {
      \href{http://www.ams.org/mathscinet-getitem?mr=#1}
    {\nolinkurl{#1}}
    }
}
\usepackage{xpatch}
\xapptobibmacro{finentry}{\setunit{\par} \printfield{mrnumber}}{}{}

\title{
Cartier--Dieudonné theory of connected group schemes
}
\author{Casimir Kothari} 
\address{Department of Mathematics, Boston University, Boston, MA}
\email{ckothari@bu.edu}
\author{Joshua Mundinger}
\address{Department of Mathematics, University of California Berkeley, Berkeley, CA}
\email{mundinger@berkeley.edu}
\date{\today}


\begin{document}

\maketitle 
\begin{abstract}

    We prove that fiberwise-connected finite locally free commutative group schemes over a ring $R$ are equivalent to an explicit full subcategory of the derived category of modules over the Cartier ring of $R$.
\end{abstract}

\section{Introduction}\label{section: intro}

Given a commutative ring $R$, let $C(R)$ be the category of finite locally free commutative group schemes over $\Spec(R)$ whose fibers are connected. The goal of this paper is to describe the category $C(R)$ via two-term complexes in the derived category of Cartier modules over $R$.

Suppose that $G\in C(R)$ has order $N > 1$; then $N$ is nilpotent in $R$.
For if $A$ is the coordinate ring of $G$ with augmentation ideal $I$, then $I$ is nilpotent.
By Deligne's theorem \cite[pp. 4]{TO70}, $[N]G = 0$, so $[N]^*I = 0$.
On the other hand, $[N]^*$ acts by multiplication by $N$ on the associated graded $I/I^2$, so some power of $N$ acts by zero on $I$. Since $I$ is locally free, $N$ is nilpotent in $R$.

By the Chinese Remainder Theorem, it suffices to fix a prime $p$ and work over a $p$-nilpotent ring $R$.
To state the main theorem, let $\What$ be the $p$-typical formal Witt group over $R$ and let $\E_R = \End(\What)^{\op}$ be the ($p$-typical) Cartier ring. To present $\E_R$ explicitly, let $D_R$ be the ring generated by $W(R)$ and the symbols $F$ and $V$, subject to the relations 
$FV=p$ and $Fa = \sigma(a)F, aV = V \sigma(a), VaF = v(a)$ for all $a \in W(R)$, where $\sigma$ is the Witt vector Frobenius and $v$ is the Witt vector Verschiebung.
Then $\E_R$ is the completion of $D_R$ with respect to the filtration $\{V^n D_R\}_{n \geq 1}$.
$\E_R$ consists of infinite sums of the form 
\[ \sum_{m,n \geq 0} V^n [r_{nm}]F^m\]
where $r_{nm} \in R$, and for fixed $n$, all but finitely many of $r_{nm}$ are zero.  The ring $\E_R$ is used in $p$-typical Cartier theory, which provides an equivalence of categories between commutative formal Lie groups over $R$ and $V$-reduced $V$-flat left $\E_R$-modules via the functor $\Hom(\What, -)$.

To state our main theorem, we recall the notion of a $V$-divided Cartier module over an $\F_p$-algebra $\bR$, as defined in \cite{Zin84}. If $N$ is a left $\E_\bR$-module, then the Frobenius twist of $N$ is $N^{(p)} := \E_\bR \ten_{\sigma, \E_\bR} N$, where $\sigma$ is the Frobenius on $\E_\bR$.  There is an $\E_\bR$-linear Frobenius map $N \to N^{(p)}$, defined by $n \mapsto V \ten n$, and the $V$-divided Cartier module of $N$ is the left $\E_\bR$-module
\be
\Div(N) := \colim N^{(p^i)}
\ee
where the colimit is taken along the Frobenius maps $N^{(p^i)} \to N^{(p^{i+1})}$. The $V$-divided Cartier module functor $\Div$ is right exact and therefore admits a left derived functor $L\Div: D(\E_\bR) \to D(\E_\bR)$.

\begin{definition}
    The \emph{Cartier--Dieudonné functor} over $R$ is the functor $\uM: C(R) \to D(\E_R)$ defined by
    \[ \uM(-):= \tau^{\leq 1}R\Hom_{\fpqc}(\What,-).\]
\end{definition}
\pagebreak
\begin{theoremalpha} \label{theorem: main}
    Let $R$ be a $p$-nilpotent ring. The Cartier--Dieudonné functor over $R$
    is fully faithful.
    Its essential image consists of those $M \in D(\E_R)$ such that 
    \begin{enumerate}
        \item $M$ is derived $V$-complete; 
        \item $M/^{\mathbb{L}}VM$ is a perfect complex of $R$-modules of tor-amplitude $[0,1]$;
        \item $L\Div(\E_{R/pR} \Lotimes{\E_R} M) = 0$.
    \end{enumerate}
\end{theoremalpha}

The Cartier--Dieudonn\'e functor has an equivalent description in terms of formal Lie groups.  Every $G \in C(R)$ admits a functorial two-term resolution $0 \to G \to K \to L \to 0$ where $K$ and $L$ are commutative formal Lie groups over $R$ (Corollary \ref{corollary: begueri for connected}), and there is a natural isomorphism
\be
\uM(G) \cong \Fib(\Hom(\What,K) \to \Hom(\What,L))
\ee
in $D(\E_R)$ (Remark \ref{remark: alternate construction}).  From this perspective, the first two conditions of Theorem \ref{theorem: main} describe those $M \in D(\E_R)$ which are equivalent to a two-term complex of Cartier modules of formal Lie groups, while the third condition describes those two-term complexes corresponding to \textit{isogenies} of formal Lie groups.  Our work consistently makes use of the relationship between connected finite locally free group schemes and formal Lie groups. 

\begin{remark}
    In \cite{Oor74}, Oort considered classifying $C(R)$ when $R$ is a complete local noetherian ring of residue characteristic $p$. Oort also used the idea of resolving $G \in C(R)$ by commutative formal Lie groups $0\to G \to K \to L \to 0$, then applying Cartier theory to $K$ and $L$.  However, instead of considering the two-term complex $[\Hom(\What,K) \to \Hom(\What,L)]$ in $D(\E_R)$, Oort considered only the cokernel $M^{\Oort}(G) = \coker(\Hom(\What,K) \to \Hom(\What,L))$. 
    The functor $G \mapsto M^{\Oort}(G)$ is not fully faithful.
    For example, when $R$ is an $\Fp$-algebra and $G = \alpha_p$, 
    the natural map 
    \[ \End_R(\alpha_p) \to \End_{\E_R}(M^{\Oort}(\alpha_p))\]
    is identified with the $p$th power map $R \to R^p$, which is not an isomorphism if $R$ is not reduced.
    See Example \ref{ex: oort-functor} for more details.
\end{remark}

\begin{remark}
For $G \in C(R)$, there is a canonical isomorphism of $\uM(G)/^{\bb{L}}V\uM(G)$ with the Lie complex $\ell_G^\vee$ of $G$. Thus, we may view $\uM(G)$ as a $V$-adic lift of the Lie complex of $G$.  Moreover, $\uM(G)$ is an iterated extension of the Frobenius restricted complexes $\ell_G^\vee, \ell_G^\vee|_{\sigma}, \ell_G^\vee|_{\sigma^2}, \dots$ (see Proposition \ref{proposition: extension of Lie}).
\end{remark}

The structure of this article is as follows. In §\ref{section: formal Lie}, we explain background in formal Lie groups and Cartier theory. We treat formal Lie groups of countable type and the corresponding Cartier theory, which may be of independent interest. For our applications to Dieudonn\'e theory, we only need Proposition \ref{prop: ext1 from What}, which states that $\Ext^1_{\fpqc}(\What, K) = 0$ for any commutative formal Lie group $K$ over $R$. We then study isogenies of formal Lie groups and describe them Cartier-theoretically in the derived category $D(\E_R)$. In §\ref{section: connected theorem}, we study the category $C(R)$ and prove Theorem \ref{theorem: main}. §4 discusses further properties of the Cartier--Dieudonné functor such as the $V$-adic filtration, compatibility with base change, and exactness. 

\subsection*{Conventions}

Throughout this article, ``countable'' means finite or countably infinite. We work with the fpqc topology by default, although our results also work in the ``fidèlement plat, pr\'esentation dénombrable'' (fppd) topology, and our results for finite-dimensional formal Lie groups are also valid in the fppf topology.  

Unless otherwise stated, whenever $R$ is a $p$-nilpotent ring, we set $\ol{R} = R/pR$, and use an overline to denote reduction from $R$ to $\ol{R}$.

We use cohomological conventions for the derived category of modules over a ring. The shift functor $[1]$ satisfies $H^i(M[1]) = H^{i+1}(M)$.

\subsection*{Acknowledgments}
We thank Vladimir Drinfeld for explaining to us the construction of the Cartier--Dieudonn\'e functor via formal Lie groups and asking about its full faithfulness and essential image.  
Thanks to 
Ben Antieau,
Matt Emerton,
Keerthi Madapusi,
Akhil Mathew,
Shubhodip Mondal,
and Martin Olsson 
for helpful comments and conversations.

In the course of this work, J.M.\ was supported by the National Science Foundation under Award No. 2503534. Any opinions, findings, and conclusions or recommendations expressed in this material are those of the authors and do not necessarily reflect the views of the National Science Foundation.

\subsection*{Computational resource disclosure}

We originally proved Theorem \ref{theorem: quotients of FLG} without computer assistance by reduction to characteristic $p$. ChatGPT 5.6 Sol suggested a more direct proof via Lemma \ref{lem: quotient quasiregular immersion}, which we used in this article. The statement of Proposition \ref{proposition: truncation RHom} and the argument for the reduction to the case $n = 1$ were suggested by ChatGPT 5.6 Sol. We also used ChatGPT for general research assistance (e.g. literature searches and exploring proof strategies), and used ChatGPT and Claude for proofreading the manuscript. We have edited, reviewed, and verified all the content generated by these tools to ensure its accuracy.

\section{Formal Lie groups and Cartier theory} \label{section: formal Lie}

Cartier theory describes the category of formal Lie groups in terms of maps from the Witt vector formal group $\What$. However, $\What$ is not a formal Lie group in the usual sense since it is infinite-dimensional.
Thus, it is useful to have a supple theory of infinite-dimensional formal Lie groups.
Zink's approach to Cartier theory \cite{Zin84} sidesteps such geometric questions by studying formally smooth functors on nilpotent algebras. However, for questions of extensions and descent, we will need to consider representability.

In this section, we introduce formal Lie groups of countable type and state their classification in terms of Cartier theory. The underlying set-valued functor of a formal Lie group of countable type over $R$ is of the form $A \mapsto \Nil(A) \otimes_R M$ where $M$ is a countably generated projective $R$-module. This is a ``Goldilocks zone'' which allows for infinite-type objects like $\What$ but avoids certain difficulties with formally smooth ind-schemes of uncountable type.

\subsection{Based formal polydisks of countable type}

Let $R$ be a commutative ring. 
Recall that an \emph{ind-scheme} over $\Spec R$ is a functor $X: \CAlg_R \to \Set$ which is a filtered colimit of schemes along closed immersions.
Since filtered colimits in $\Set$ commute with finite limits, an ind-scheme is an fpqc sheaf. An ind-scheme is \emph{ind-affine} if it is the colimit of affine schemes along closed immersions. An ind-affine ind-scheme is \emph{ind-finite} over $\Spec R$ if those affine schemes may be taken to be finite over $R$.

Recall that a \emph{nilpotent $R$-algebra} is a non-unital commutative $R$-algebra $N$ where there exists $m > 0$ such that $N^m =0 $. If $N$ is a nilpotent $R$-algebra, then we can form the augmented nilpotent $R$-algebra $R \oplus N$, which is a unital $R$-algebra.

\begin{proposition}\label{prop: ind-infinitesimal conditions}
    Let $X \to \Spec R$ be an ind-finite ind-scheme equipped with a section $e: \Spec R \to X$. The following are equivalent:
    \begin{enumerate}
        \item for all $x \in X(A)$ there exists a finitely generated nilpotent ideal $I \subseteq A$ such that $x$ maps to $e \in X(A/I)$;
        \item for all $x \in X(A)$ there exists a finite nilpotent $R$-algebra $N$ such that $x$ factors through $X(R \oplus N)$ over $e$;
        \item $X$ is a filtered colimit $\colim_\alpha \Spec(R \oplus N_\alpha)$ where each $N_\alpha$ is a finite nilpotent $R$-algebra.
    \end{enumerate}
\end{proposition}
\begin{proof}
    Write $X = \colim_\alpha \Spec B_\alpha$ where each $B_\alpha$ is finitely generated as an $R$-module; then $X(A) = \colim_\alpha \Hom(B_\alpha, A)$ for all $A \in \CAlg_R$.
    Since our colimit is filtered, we can assume that $e: \Spec R \to X$ factors through every $\Spec B_\alpha \to X$; hence each $B_\alpha$ has an augmentation $\epsilon_\alpha: B_\alpha \to R$.

    Suppose i. If $x \in X(A)$, then there exists $\alpha$ such that $x$ factors through $\Spec B_\alpha(A)$, and so that the image of $x$ in $\Spec B_\alpha(A/I)$ is $e$ for some finite nilpotent ideal $I$.
    Thus we have a diagram 
    \[
\begin{tikzcd}
	{B_\alpha} & A \\
	R & {A/I}
	\arrow["f", from=1-1, to=1-2]
	\arrow["{\epsilon_\alpha}"', from=1-1, to=2-1]
	\arrow[from=1-2, to=2-2]
	\arrow[from=2-1, to=2-2]
\end{tikzcd}
.
    \]
    Then $f$ factors through $(\epsilon_\alpha, f - \epsilon_\alpha): B_\alpha \to R \oplus I$.
    If $N$ is the image of the augmentation ideal of $B_\alpha$ in $I$, then $N$ is nilpotent, finite over $R$, and $f$ factors through $R \oplus N$.
    Thus i. $\implies$ ii.

    Suppose ii. Given $\alpha$, the structure map $\Spec B_\alpha \to X$ factors through some $\Spec(R \oplus N)$. 
    Hence there is $\beta\geq \alpha$ such that the map $B_\beta \to B_\alpha$ factors through $R \oplus N$ compatibly with the augmentations.
    Since $B_\beta \to B_\alpha$ is surjective, this implies $B_\alpha = R \oplus N_\alpha$ where $N_\alpha$ is the image of $N$. Therefore $N_\alpha$ is a finite nilpotent ideal. Thus ii. $\implies$ iii.

    Suppose iii. Given $x \in X(A)$, there exists $\alpha$ such that $x$ factors through $\Spec(R\oplus N_\alpha) \to X$.
    Then set $I = AN_\alpha$. Thus iii. $\implies$ i.
\end{proof}

\begin{definition}
    A pointed ind-scheme $(X,e)$ is \emph{ind-infinitesimal} if it is ind-finite and satisfies the equivalent conditions of Proposition \ref{prop: ind-infinitesimal conditions}. 
\end{definition}

\begin{definition}
    Let $R$ be a commutative ring. 
    A \emph{based formal polydisk of countable type} (or \emph{based $\alepho$-polydisk}) over $R$ is a pointed ind-scheme $(X,e) \to \Spec R$
    such that 
    \begin{enumerate}
        \item $X \to \Spec R$ is an ind-finite $\alepho$-ind-scheme, that is, a countable colimit of schemes along closed immersions;
        \item $(X,e) \to \Spec R$ is ind-infinitesimal;
        \item $X \to \Spec R$ is formally smooth.
    \end{enumerate}
    Let $\aPolyd(R)$ be the groupoid of based $\alepho$-polydisks over $R$ and their isomorphisms.
\end{definition}

\begin{example}
    Let $I$ be a countable set.
    Consider the functor 
    \[\Ahat^I(A) = \{ f \in A^I \mid f(i) \text{ is nilpotent, } f(i) = 0 \text{ for all but finitely many }i\};\]
    The functor is formally smooth.
    Let $e = 0 \in \Ahat^I(R)$.
    Then $\Ahat^I$ is the filtered colimit 
    \[ \colim_{J \subset I \text{finite}} \colim_{n \in \mathbb{N}} \Spec(R[x_j \mid j \in J]/(x_j \mid j \in J)^n),\]
    and thus is an ind-infinitesimal ind-scheme.
    The system of finite subsets of $I$ is countable, so $\Ahat^I$ is $\alepho$. Thus $(\Ahat^I,0)$ is a based $\alepho$-polydisk.
\end{example}

Note that the example $\Ahat^I$ is of the form $A \mapsto \Nil(A) \otimes_R M$ where $M = \bigoplus_I R$ is a free $R$-module.
\begin{definition}\label{defn: ahat}
    For an $R$-module $M$, define $\Ahat_M: \CAlg_R \to \Set$ by 
    \[ \Ahat_M(A) = \Nil(A) \otimes_R M.\]
\end{definition}

Theorem \ref{theorem: formal alepho-Lie varieties} below shows that every based $\alepho$-polydisk is of the form $\Ahat_M$ where $M$ is a countably generated projective module.
To prove this, we use the characterization of countably generated projective modules as \emph{flat Mittag-Leffler modules}, which is more suited to the geometry of ind-schemes.
The notion of flat Mittag-Leffler module was introduced by Raynaud and Gruson in their study of fpqc descent for projective modules \cite{RG71}.
Recall Lazard's theorem that an $R$-module is flat if and only if it is a filtered colimit of finite free modules.

\begin{definition}
    An $R$-module $M$ is \emph{flat Mittag-Leffler}
    if when $M = \colim_\alpha F_\alpha$ for finite free $R$-modules $F_\alpha$, the pro-system of dual modules $F_\alpha^*$ is Mittag-Leffler:
    for all $\alpha$ there exists $\beta \geq \alpha$ such that for $\gamma \geq \beta$,
    \[ \im(F_\gamma^* \to F_\alpha^*) = \im(F_\beta^* \to F_\alpha^*).\]
\end{definition}
The Mittag-Leffler condition on $F_\alpha^*$ means that the pro-system $F_\alpha^*$ is equivalent to a pro-system of $R$-modules with surjective transition maps (take $M_\alpha$ to be the stable image of $F_\beta^* \to F_\alpha^*$).

\begin{theorem}[Raynaud--Gruson, \cite{stacks-project}, \href{https://stacks.math.columbia.edu/tag/059X}{Tag 059X}]\label{theorem: countably generated projective}
    For a countably generated $R$-module $M$, $M$ is flat Mittag-Leffler if and only if $M$ is projective.
\end{theorem}

The following proposition is useful for identifying flat Mittag-Leffler modules.

\begin{proposition}[\cite{BD91_Hitchin}, Proposition 7.12.6; \cite{AM25}, Theorem 2.3.1]
    \label{prop: BD criterion for FML}
    Consider functors $R\modc \to R\modc$ of the form $M \otimes_R -$ for an $R$-module $M$ and of the form $\colim_\alpha \Hom(N_\alpha,-)$ for a filtered system of $R$-modules $N_\alpha$.
    \begin{enumerate}
        \item A functor of the form $M \otimes_R -$ is isomorphic to one of the form $\colim_\alpha \Hom(N_\alpha, -)$ if and only if $M$ is flat.
        \item A functor of the form $M \otimes_R -$ is isomorphic to one of the form $\colim_\alpha \Hom(N_\alpha, -)$ with surjective transition maps $N_\beta \to N_\alpha$ if and only if $M$ is flat Mittag-Leffler.
        \item A functor of the form $\colim_\alpha \Hom(N_\alpha ,-)$ with surjective transition maps $N_\beta \to N_\alpha$ is isomorphic to one of the form $M \otimes_R -$ if and only if the functor is exact and the modules $N_\alpha$ are finitely generated.
    \end{enumerate}
\end{proposition}

\begin{remark}
    There is a more general notion of a Mittag-Leffler module without flatness, but we only require the notion of flat Mittag-Leffler module.
\end{remark}

Recall that if $X: \CAlg_R \to \Set$ is a functor and $x \in X(A)$ is a point of $X$, then the \emph{tangent space} to $x$ is the functor $T_x X: A\modc \to \Set$ defined by $T_xX(M) = X(A \oplus M) \times_{X(A)} \{x\}$.

\begin{lemma}[cf. \cite{BD91_Hitchin}, 7.12.12]
    \label{lemma: tangent space is ML}
    Suppose that $X\to \Spec R$ is a formally smooth ind-scheme of ind-finite type and $x \in X(A)$.
    Then $T_xX(A)$ is a flat Mittag-Leffler $A$-module and $T_xX(-) = T_xX(A) \otimes_A -$.
\end{lemma}
\begin{proof}
    Write $X = \colim_\alpha X_\alpha$ for $X_\alpha \to \Spec R$ schemes of finite type.
    Then $T_x X(-) = \colim_\alpha \Hom(x^* \Omega^1_{X_\alpha/R}, -)$. The $A$-modules $x^*\Omega^1_{X_\alpha/R}$ are finitely generated and the transition maps are surjective. Since $X$ is formally smooth, the functor $T_xX$ is exact. 
    By Proposition \ref{prop: BD criterion for FML},
    $T_xX$ is given by tensoring with a flat Mittag-Leffler $A$-module, which must be $T_xX(A)$.
\end{proof}

We can now prove a classification of based $\alepho$-polydisks in terms of countably generated projective modules. 

\begin{theorem}\label{theorem: formal alepho-Lie varieties}
    \begin{enumerate}
        \item If $M$ is an $R$-module, then $(\Ahat_M,0)$ is a based $\alepho$-polydisk if and only if $M$ is countably generated projective.
        \item (\cite[Proposition 7.12.18]{BD91_Hitchin}) Suppose $(X,e) \to \Spec R$ is a based $\alepho$-polydisk. Then 
        \[(X,e) \cong (\Ahat_M,0)\]
        where $M = T_eX(R)$.
    \end{enumerate}
\end{theorem}
\begin{proof}
    \begin{enumerate}
        \item First suppose that $(\Ahat_M,0)$ is a based $\alepho$-polydisk.
        By definition, $M = T_0 \Ahat_M$.
        By Lemma \ref{lemma: tangent space is ML}, $M$ is flat Mittag-Leffler; since $\Ahat_M$ is ind-infinitesimal and $\alepho$, $T_0\Ahat_M(-) = \colim_\alpha \Hom(N_\alpha,-)$ where the modules $N_\alpha$ are finitely presented and the colimit is countable. If we write $M$ as a colimit of finite free modules $F_\beta$, then the pro-systems $\{N_\alpha\}$ and $\{F_\beta^*\}$ are equivalent; the system of $N_\alpha$'s is countable, so we may refine the system of $F_\beta^\ast$'s to be countable. Thus $M$ is countably generated. By Theorem \ref{theorem: countably generated projective}, $M$ is countably generated projective.

        Conversely, if $M$ is countably generated projective, write $M = \colim_\alpha F_\alpha$ where each $F_\alpha$ is finite free and the colimit is countable. Since the pro-system $F_\alpha^*$ is Mittag-Leffler, by passing to the stable images $N_\alpha$, we obtain an equivalent pro-system of finitely generated modules with surjective transition maps.
        Then 
        \begin{align*} 
            \Ahat_M(A) &\cong \colim_\alpha \Hom_R(N_\alpha, \Nil(A)) \\
            &=\colim_\alpha \colim_{n\in\N} \Hom_{\CAlg_R} (\Sym(N_\alpha)/\Sym(N_\alpha)_{> n}, A)
        \end{align*}
        since each $N_\alpha$ is finitely generated.
        This exhibits $\Ahat_M$ as an ind-infinitesimal $\alepho$-ind-scheme. It is formally smooth since $M$ is flat.

        \item Suppose $(X,e)$ is a based $\alepho$-polydisk.
        Set $M = T_eX(R)$.
        We begin by constructing a suitable map $(X,e) \to (\Ahat_M,0)$.
        Write $(X,e) = \colim_{\alpha \in I} (X_\alpha,e_\alpha)$ where $I$ is countable.
        Then 
        \[\Hom((X,e),(\Ahat_M,0)) = \ilim_{\alpha \in I^{op}} \Hom((X_\alpha,e_\alpha), (\Ahat_M,0)).\]

        If $(X_\alpha,e_\alpha) = \Spec(R \oplus N_\alpha)$ where $N_\alpha$ is nilpotent, then 
        \[ T_eX = \colim_\alpha \Hom(N_\alpha/N_\alpha^2,-),\]
        so there is a tautological element $\iota_\alpha \in T_eX(N_\alpha/N_\alpha^2) = N_\alpha/N_\alpha^2 \otimes M$ corresponding to the identity endomorphism of $N_\alpha/N_\alpha^2$.
        Define the inverse system of sets
        \begin{align*} 
            S_\alpha &= \Ahat_M(R \oplus N_\alpha) \underset{\Ahat_M(R \oplus N_\alpha/N_\alpha^2)}{\times} \{\iota_\alpha\}\\
            &= \{f \in N_\alpha \otimes_R M \mid f \mapsto \iota_\alpha \in N_\alpha/N_\alpha^2 \otimes_R M\}.
        \end{align*}
        Since $M$ is flat, $\Ahat_M$ is formally smooth, so each $S_\alpha$ is nonempty. Moreover, the transition maps $X_\alpha \to X_\beta$ are closed immersions defined by nilpotent ideals and $\Ahat_M$ is formally smooth, so the transition maps $S_\beta \to S_\alpha$ are surjective.
        Since $I$ is countable, $\ilim_{\alpha \in I^{op}} S_\alpha$ is nonempty, and an element of this inverse limit defines a pointed map $f:X \to \Ahat_M$ such that $D_ef: M = T_eX \to T_e\Ahat_M = M$ is the identity endomorphism of $M$.
        We claim $f$ is an isomorphism. As both $X$ and $\Ahat_M$ are ind-infinitesimal ind-schemes, it suffices to verify that $f$ induces an isomorphism $X(A) \to \Ahat_M(A)$ whenever $A = R \oplus N$ for nilpotent $N$.
        But this follows from a deformation theory argument, since both $X$ and $\Ahat_M$ send pullbacks of rings to pullbacks and $f$ is an isomorphism on tangent spaces.\qedhere
    \end{enumerate}
\end{proof}

\begin{remark}
    There are based polydisks of uncountable type not of the form $\Ahat_M$ for any module $M$. See \cite[Example 7.12.19]{BD91_Hitchin}.
\end{remark}

Our next task is to deal with fpqc descent of based $\alepho$-polydisks.
It is known that finite-dimensional formal based polydisks form an fpqc stack \cite[§3.1]{Dri24}.
The proof uses that the tangent space to a finite-dimensional polydisk is a finite locally free module.
In the countable setting, it is no longer true that a countably generated projective module is locally free (see e.g. \cite[\href{https://stacks.math.columbia.edu/tag/05WG}{Tag 05WG}]{stacks-project}), even in the fpqc topology. Thus a different approach is needed. 

Recall that an ind-affine ind-scheme $X\to \Spec R$ is said to be \emph{coflat} if its pro-coordinate ring, as a pro-$R$-module, is equivalent to a pro-system of projective modules \cite[Definition 2.5.7]{AM25}. If $M= \colim_i F_i$ is a flat Mittag-Leffler $R$-module, then the ind-scheme $\Ahat_M$ is coflat, since its pro-coordinate ring is $\ilim_i \widehat{\Sym} F_i^*$. 
Thus every based $\alepho$-polydisk is coflat.

\begin{proposition}\label{prop: coalgebras}
    The category of $\alepho$ ind-finite ind-schemes coflat over $R$ is equivalent to the category of cocommutative coalgebras $C$ over $R$ such that $C$ is a countably generated projective $R$-module.
    The equivalence sends $C$ to 
    \[ \coSpec(C) (A) = \{ x \in A \otimes_R C \mid \Delta x = x \otimes x \text{ in } A \otimes_R (C \otimes_R C), \epsilon(x) = 1 \text{ in } A\}.\]
\end{proposition}
\begin{proof}
    By \cite[Proposition 2.5.5]{AM25}, $\alepho$ coflat ind-finite ind-schemes are equivalent to pro-finite coflat Mittag-Leffler pro-$R$-algebras $A$ whose underlying pro-module is $\alepho$, where the equivalence sends $A$ to $\SpInd(A) = \Hom_R(A,-):\CAlg_R\to \Set$.
    Under linear duality $A \mapsto A^\vee$, \cite[Proposition 2.4.1]{AM25}, such pro-$R$-algebras are exactly dual to countably generated projective $R$-coalgebras,
    and $\SpInd(A) = \coSpec(A^\vee)$.
\end{proof}

\begin{theorem}\label{theorem: apolyd is a stack}
    $R \mapsto \aPolyd(R)$ is a stack in the fpqc topology.
\end{theorem}
\begin{proof} 
    Suppose $R \to R'$ is a faithfully flat morphism of rings, and $p_1,p_2: R' \to R' \otimes_R R'$ the standard morphisms.
    We must show that if $(X',e') \to \Spec R'$ is a based $\alepho$-polydisk equipped with an isomorphism $p_1^* (X',e') \cong p_2^*(X',e')$ satisfying the cocycle condition, then $(X',e')$ is the pullback of a based $\alepho$-polydisk from $\Spec R$.

    By Proposition \ref{prop: coalgebras}, $(X',e') = \coSpec(C')$ where $C'$ is a coaugmented cocommutative coalgebra over $R'$; moreover, we have isomorphisms $p_1^*C' \cong p_2^*C'$ satisfying the cocycle condition. Hence $C' = C \otimes_R R'$ where $C$ is a coaugmented cocommutative coalgebra over $R$. By flat descent for countably generated projective modules (Raynaud--Gruson, \cite[\href{https://stacks.math.columbia.edu/tag/05A9}{Tag 05A9}]{stacks-project}), $C$ is a countably generated projective $R$-module. 
    
    We claim that $(\coSpec(C),e)$ is a based $\alepho$-polydisk.
    Since $R \to R'$ is faithfully flat, $\Prim(C') = \Prim(C) \otimes_R R'$. By Theorem \ref{theorem: formal alepho-Lie varieties}, $\Prim(C')$ is a countably generated projective $R'$-module; again by Raynaud and Gruson's theorem, $M = \Prim(C)$ is a countably generated projective $R$-module.
    Since $\coSpec(C')$ is ind-infinitesimal, $C'$ is conilpotent, so $C$ is also conilpotent.
    Thus $\coSpec(C)$ is ind-infinitesimal.
    Since $\Ahat_M$ is formally smooth, copying the argument of Theorem \ref{theorem: formal alepho-Lie varieties} (ii) shows that there exists a morphism $f: (\coSpec(C),e) \to (\Ahat_M,0)$ inducing an isomorphism on tangent spaces.
    The base change $f \times_{\Spec R} \Spec R'$ is an isomorphism.
    By Proposition \ref{prop: coalgebras}, it suffices to show $f$ induces an isomorphism on the level of coalgebras, which follows from faithfully flat descent along $R \to R'$. Thus, $f$ is an isomorphism and $(\coSpec(C),e)$ is a based $\alepho$-polydisk.
\end{proof}

\subsection{Formal Lie groups of countable type}

\begin{definition}
    A \emph{formal Lie group of countable type}, or \emph{formal $\alepho$-Lie group}, over $R$ is a functor from $R$-algebras to groups whose underlying set-valued functor is a based $\alepho$-polydisk.
    Let $\aFLG$ be the category of formal $\alepho$-Lie groups, and let $\aCFLG(R)$ be the category of commutative formal $\alepho$-Lie groups over $R$.
\end{definition}

If $E \to \Spec R$ is a commutative formal $\alepho$-Lie group, then $E$ is a coflat ind-finite ind-scheme, so $E$ is reflexive with respect to Cartier duality \cite[Corollary 3.1.6]{AM25}.

\begin{lemma}\label{lemma: restriction of fgl-torsors on affines}
    Suppose that $E \to \Spec R$ is a commutative formal $\alepho$-Lie group and 
    \[ 0 \to I \to \tilde S \to S \to 0\]
    is an extension of rings where $I$ is nilpotent.
    Then the restriction map $H^1_{\fpqc}(\Spec \tilde S,E) \to H^1_{\fpqc}(\Spec S, E)$ of cohomology on the small fpqc sites is injective.
\end{lemma}
\begin{proof}
    By induction, we may assume $I^2 = 0$. Let $X = \Spec S, \tilde X = \Spec \tilde S$, $i: X \to \tilde X$ the closed immersion defined by $I$.
    The restriction map 
    $H^1_{\fpqc}(\tilde X, E) \to H^1_{\fpqc}(X, E)$ factors as 
    \[ 
        H^1_{\fpqc}(\tilde X, E_{\tilde X}) \to H^1_{\fpqc}(\tilde X, i_*E_{X}) \to H^1_{\fpqc}(X, E_X)
    \]
    where $i_*$ is the pushforward of fpqc abelian sheaves.
    The five-term exact sequence from the Grothendieck spectral sequence associated to $R\Gamma_{\fpqc}(\tilde X,Ri_*(-)) = R\Gamma_{\fpqc}(X,-)$
    for $E_X$
    begins 
    \[ 0 \to H^1_{\fpqc}(\tilde X, i_*E_X) \to H^1_{\fpqc}(X,E_X) \to H^0_{\fpqc}(\tilde X, R^1i_*E_X) \to \cdots\]
    so $H^1_{\fpqc}(\tilde X, i_*E_X) \to H^1_{\fpqc}(X,E_X)$ is injective.
    Thus it suffices to show 
    \[ H^1_{\fpqc}(\tilde X, E_{\tilde X}) \to H^1_{\fpqc}(\tilde X, i_*E_X)\]
    is injective. If $M$ is the tangent space of $E$ over $R$,
    then for any $\tilde S$-algebra $A$ there is a short exact sequence 
    \[ 0 \to M \otimes_{R} (IA) \to E(A) \to E(A/IA) \to 0\]
    where exactness on the right is since $E$ is formally smooth.
    When $A$ is flat over $\tilde{S}$, we have $M \otimes_R (IA) = (M \otimes_R I)\otimes_{\tilde S} A$, so there is an exact sequence of sheaves
    \[ 0 \to \underline{M \otimes_R I} \to E_{\tilde S} \to i_*E_S \to 0\]
    on the small fpqc site of $\tilde X$. By fpqc descent of quasicoherent sheaves and Serre's theorem, $H^1_{\fpqc}(\tilde X, \underline{M \otimes_R I}) = 0$, which proves the claim.
\end{proof}

\begin{theorem}\label{theorem: alepho-Lie closed under extensions}
    Let $E'$ and $E''$ be commutative formal $\alepho$-Lie groups over $R$ and 
    \[ 0 \to E' \to E \to E'' \to 0\]
    be a short exact sequence of abelian fpqc sheaves.
    Then $E$ is a formal $\alepho$-Lie group.
\end{theorem}
\begin{proof}
    First, we show $E \to E''$ has a set-theoretic section.
    Note that $E$ is an $E'$-torsor over $E''$ with a trivialization at $0 \in E''$.
    Since $E''$ is an $\alepho$-Lie group, it is a countable colimit $\colim_\alpha (X_\alpha,e_\alpha)$ of finite $R$-schemes along nil-immersions. By Lemma \ref{lemma: restriction of fgl-torsors on affines}, $E \times_{E''} X_\alpha$ is a trivial torsor.
    Let $S_\alpha$ be the set of sections of $E \times_{E''} X_\alpha \to X_\alpha$ whose restriction along $e_{\alpha}$ is $0 \in E'$.
    Since $E'$ is formally smooth, the transition maps $S_\alpha \to S_\beta$ are surjective. Moreover, the pro-system of $S_\alpha$'s is countable, so there exists an element of the inverse limit, giving a section of $E \to E''$.
    
    The section $\sigma$ shows that the underlying pointed set-valued functor of $(E,0)$ is isomorphic to $(E' \times E'', 0 \times 0)$. 
    Since the product of based $\alepho$-polydisks is an $\alepho$-polydisk, $E$ is a formal $\alepho$-Lie group.
\end{proof}

In this article, we use the term formal Lie group to refer specifically to formal $\alepho$-Lie groups of finite type: 
\begin{definition}
    A \emph{formal Lie group} is a group object in based polydisks over $R$, that is, in $\alepho$-polydisks whose tangent space is a finite locally free $R$-module.
\end{definition}
This definition essentially appears in \cite{Dri24}. See \cite[Chapter II, Definition 1.1.5]{Mes72} for an equivalent formulation of this notion.

\subsection{Cartier theory in countable type}

Cartier theory is an equivalence of categories between commutative ``formal Lie groups'' and certain modules over the Cartier ring $\E_R$ from §1. Cartier's original papers \cite{Car67Groupes, Car67Modules}
treat formal Lie groups whose underlying ind-scheme is $\Ahat^r$ for some $r$, as does Hazewinkel's monograph \cite[Chapter V]{Haz78}. Zink's treatment of Cartier theory in \cite{Zin84} both relaxes the requirement for global coordinates and extends the theory beyond the finite-dimensional case. However, Zink's treatment uses slightly different language, generally avoiding geometric objects like ind-schemes. In this section, we provide the translation between Zink's language and our own, and give a Cartier theory for formal $\alepho$-Lie groups by quoting Zink's results. 
As a corollary, we give vanishing results for the group $\Ext^1_{\fpqc}(\What,-)$ (Proposition \ref{prop: ext1 from What}). This is the main reason for our treatment of formal $\alepho$-Lie groups, and will be used in the sequel for calculating $\tau^{\leq 1}R\Hom_{\fpqc}(\What,-)$.

Given a ring $R$, let $\Nil_R$ be the category of nilpotent $R$-algebras. 
\begin{definition}\label{def: Zink formal group}
    A \emph{Zink commutative formal group} is a functor $\Nil_R \to \Ab$ which is exact and commutes with infinite direct sums.
\end{definition}

Let $\aug: \Nil_R \to (\CAlg_R)_{/R}$ be the functor taking a nilpotent $R$-algebra $N$ to the augmented $R$-algebra $R\oplus N \to R$.
If $X: \CAlg_R \to \Set$ is a functor equipped with a section $e: \Spec R \to X$, then $\aug^*X: \Nil_R \to \Set$ is the functor taking $N$ to $X(R \oplus N) \times_{X(R)} \{e\}$.

\begin{proposition}\label{prop: comparison to Zink FLG}
    Let $R$ be a commutative ring.
    Then sending $X$ to $\aug^*X$ is an equivalence of categories 
    from commutative formal $\alepho$-Lie groups over $R$ to Zink commutative formal groups over $R$ whose tangent space to the identity is a countably generated projective $R$-module.
\end{proposition}
\begin{proof}
    First, we show $X \mapsto \aug^*X$ is fully faithful.
    It suffices to show that $(X,e) \mapsto \aug^*X$ is a fully faithful functor from ind-infinitesimal ind-schemes to $\Fun(\Nil_R,\Set)$.
    This follows since $(X,e)$ as a pointed functor is a colimit of functors represented by augmented nilpotent algebras.
    
    Now suppose that $Y$ is a Zink commutative formal group over $R$ whose tangent space $M$ is countably generated projective. By \cite[Theorem 2.31]{Zin84}, $Y$ is prorepresentable, and indeed the proof shows that the underlying set of $Y$ is of the form $\aug^*X$ where $X = \Ahat_M$ is the based $\alepho$-polydisk.
    Since $\aug^*$ is fully faithful on ind-infinitesimal ind-schemes, the group structure on $Y$ transfers to $X$.
\end{proof}

We recall Zink's version of the main theorems of $p$-typical Cartier theory.
If $R$ is a $\Z_{(p)}$-algebra, 
let $\E_R$ be the Cartier ring, as defined in §\ref{section: intro}.
\begin{definition}\label{def: V-reduced}
    An $\E_R$-module $M$ is \emph{$V$-reduced} if 
    \begin{itemize}
        \item $V: M \to M$ is injective, and 
        \item $M \to \ilim_n M/V^nM$ is an isomorphism.
    \end{itemize}
\end{definition}

\begin{theorem}[\cite{Zin84}, Theorem 4.23]\label{theorem: Zink's cartier equivalence}
    Let $R$ be a $\Z_{(p)}$-algebra. The functor 
    $\Hom(\What,-)$
    is an equivalence from Zink commutative formal groups 
    to $V$-reduced $\E_R$-modules such that $M/VM$ is a flat $R$-module.
\end{theorem}

Write $M(-) = \Hom(\What,-)$ for the equivalence of Theorem \ref{theorem: Zink's cartier equivalence}; we call this the \textit{Cartier functor.} Under this equivalence, the tangent space of a Zink commutative formal group $K$ is given by $M(K)/VM(K)$ (\cite[Lemma 4.18]{Zin84}).

\begin{theorem}\label{theorem: cartier in countable type}
    Let $R$ be a $\Z_{(p)}$-algebra. The functor 
    \[ M(-) = \Hom(\What, -): \aCFLG(R) \to \E_R\modc\]
    is fully faithful. Its essential image consists of 
    $V$-reduced $\E_R$-modules such that $M/VM$ is a countably generated projective $R$-module.
\end{theorem}
\begin{proof}
    By Proposition \ref{prop: comparison to Zink FLG}, $\aCFLG(R)$ is equivalent to those Zink formal groups whose tangent space is countably generated projective. Under Theorem \ref{theorem: Zink's cartier equivalence}, these correspond to $V$-reduced modules $M$ such that $M/VM$ is countably generated projective.
\end{proof}

When $M$ is $V$-reduced and $M/VM$ is a finite free $R$-module, there is a particularly nice resolution of $M$ called the \emph{structure equations}.

\begin{proposition}[\cite{Zin84}, Theorem 4.39]
\label{prop: structure equations}
    Suppose that $\alpha_{i,n,j} \in W(R)$ for $1\leq i,j\leq d$ and $n \geq 0$. Then the map
        \[ \varphi: \E_R^{d} \to \E_R^d,\]
        \[ \varphi(e_i) = Fe_i - \sum_{j = 1}^d\sum_{n \geq 0} V^n \alpha_{i,n,j} e_j\]
        is injective, $M=\coker(\varphi)$ is $V$-reduced, and $M/VM$ is a free $R$-module of rank $d$.
        
    Further, suppose that $M$ is a $V$-reduced $\E_R$-module and $M/VM$ is a free $R$-module of rank $d$.
        Then there exists a short exact sequence 
        \[ 0 \to \E_R^d \overset{\varphi}{\to} \E_R^d \to M \to 0\]
        where $\varphi$ has the form 
        \[ \varphi(e_i) = Fe_i - \sum_{j=1}^d \sum_{n\geq 0} V^n [c_{i,n,j}] e_j\]
        for $c_{i,n,j} \in R$.
\end{proposition}
\begin{proof}
    The first claim is \cite[Theorem 4.39]{Zin84}, while the second claim follows from the last statement of \cite[Theorem 4.39]{Zin84} and the text immediately preceding it. 
\end{proof}

\begin{proposition}\label{prop: ext1 from What}
    Suppose that $K$ is a commutative formal $\alepho$-Lie group over a $\Z_{(p)}$-algebra $R$.
    Then 
    \[\Ext^1_{\fpqc}(\What,K) = 0.\]
\end{proposition}
\begin{proof}
    Suppose that 
    \begin{equation}\label{eq: extension of What}
        0 \to K \to E \to \What \to 0
    \end{equation} 
    is an extension of fpqc abelian sheaves.     
    Applying the Cartier functor $M(-) = \Hom(\What,-)$ to \eqref{eq: extension of What}, we obtain a left-exact sequence 
    \[ 0 \to M(K) \to M(E) \to M(\What).\]
    We claim that the map $M(E) \to M(\What)$ is surjective.
    For $\Hom(\What,-)$ represents the functor of $p$-typical curves, which is a summand of the functor of based curves $\Maps(\widehat{\mathbb{A}}^1,-)$.
    By Theorem \ref{theorem: alepho-Lie closed under extensions}, $E$ is an $\alepho$-Lie group.
    From the proof of Theorem \ref{theorem: alepho-Lie closed under extensions}, $E \to \What$ is in fact a relative based $\alepho$-polydisk, hence is formally smooth. Thus, any based morphism $\widehat{\mathbb{A}}^1 \to \What$ lifts to $E$.
    
    Hence we have a short exact sequence 
    \[ 0 \to M(K) \to M(E) \to M(\What) \to 0.\]
    But $M(\What) = \E_R$ is a projective $\E_R$-module,
    so there is a section of this sequence. 
    By definition of $M$, this means there is a section of $E \to \What$.
    Thus $\Ext^1_{\fpqc}(\What,K) = 0$.
\end{proof}

\begin{corollary}\label{corollary: M is exact on FLGs}
    Suppose that 
    \[ 0 \to K' \to K \to K''\to 0\]
    is a sequence of commutative formal $\alepho$-Lie groups over $R$ which is exact in the fpqc topology.
    Then 
    \[0 \to M(K') \to M(K) \to M(K'') \to 0\]
    is exact.
\end{corollary}
\begin{proof}
    By Proposition \ref{prop: ext1 from What}, $\Ext^1_{\fpqc}(\What, K') = 0$.
\end{proof}

Corollary \ref{corollary: M is exact on FLGs} shows that an extension of formal $\alepho$-Lie groups defines an extension of their Cartier modules.

\begin{lemma}\label{lemma: cartier functor on ext1}
    Let $R$ be a $\Z_{(p)}$-algebra and $K$ and $L$ be commutative formal $\alepho$-Lie groups over $R$.
    Then sending an extension of abelian fpqc sheaves
    \[ 0 \to L \to E \to K \to 0\]
    to 
    \[ 0 \to M(L) \to M(E) \to M(K) \to 0\]
    defines an isomorphism 
    \[ \tilde M: \Ext^1_{\fpqc}(K,L) \overset{\cong}{\to} \Ext^1_{\E_R}(M(K),M(L)).\]
\end{lemma}
\begin{proof}
    By Theorem \ref{theorem: alepho-Lie closed under extensions}, $E$ is a formal $\alepho$-Lie group. By Corollary \ref{corollary: M is exact on FLGs}, $0 \to M(L) \to M(E) \to M(K) \to 0$ is exact, so $\tilde M$ is well-defined.
    
    First, we show $\tilde M$ is injective. Suppose 
    \[ 0 \to M(L) \to M(E) \to M(K) \to 0\]
    is split. By Cartier theory (Theorem \ref{theorem: cartier in countable type}), the Cartier functor $M(-)$ is fully faithful on commutative formal $\alepho$-Lie groups, 
    so if $f: M(K) \to M(E)$ is a splitting of $M(q): M(E) \to M(K)$, then $f = M(\psi)$ where $\psi$ is a splitting of $q$. Thus $\tilde M$ is injective.
    
    Now we show $\tilde M$ is surjective. We first claim that if 
    \begin{equation} \label{eq: ext1 of Cartier modules}
        0 \to M(L) \to M \to M(K) \to 0
    \end{equation}
    is an extension of $\E_R$-modules, then $M$ is the Cartier module of a formal $\alepho$-Lie group. 
    The Snake Lemma shows that $V: M \to M$ is injective
    and that there are short exact sequences 
    \[ 0 \to M(L)/V^nM(L) \to M/V^nM \to M(K)/V^nM(K)\to 0\]
    for all $n$.
    Since the inverse system $\{M(L)/V^nM(L)\}_{n \geq 1}$ satisfies the Mittag-Leffler condition, passing to inverse limits gives a short exact sequence
    \[ 0 \to M(L) \to \ilim_n M/V^nM \to M(K) \to 0.\]
    Comparing this to \eqref{eq: ext1 of Cartier modules} shows $M \to \ilim_n M/V^nM$ is an isomorphism. Thus $M$ is $V$-reduced.
    The short exact sequence 
    \[ 0 \to M(L)/VM(L) \to M/VM \to M(K)/VM(K) \to 0\]
    shows that $M/VM$ is a countably generated projective $R$-module, and thus $M$ is the Cartier module of a formal $\alepho$-Lie group.
   
    If $M = M(E)$, then by full faithfulness of the Cartier functor $M(-)$ we obtain maps $L \to E$ and $E \to K$ inducing \eqref{eq: ext1 of Cartier modules}.
    The inverse of the Cartier equivalence is exact by \cite[Theorem 4.41]{Zin84}, so $0 \to L \to E \to K \to 0$ is exact. 
\end{proof}

\subsection{Isogenies of Formal Lie Groups}

In this subsection we define isogenies of formal Lie groups and recall their Cartier-theoretic characterization. Our definition is formulated geometrically: an isogeny is an fpqc epimorphism with finite locally free kernel. In his work on Cartier theory \cite{Zin84}, Zink uses a different definition, requiring equality of dimensions and a nilpotence condition on the kernel. We first show that the two definitions agree, so that Zink’s results on isogenies may be applied in our setting. We then explain how isogenies are detected after passing from Cartier modules to their $V$-divided Cartier modules, and prove a version of this criterion formulated in the derived category of $\E_R$-modules $D(\E_R)$.

\begin{definition}
Let $R$ be a ring.  A homomorphism $f: K \to L$ of formal Lie groups over $R$ is an \textit{isogeny} if it is surjective as a morphism of fpqc sheaves and if $\ker(f)$ is representable by a finite locally free $R$-group scheme.
\end{definition}

\begin{remark}\label{remark: fppf vs fpqc}
    If $f: K \to L$ is a morphism of formal Lie groups over $R$ with finite locally free kernel, then $f$ is an fppf epimorphism if and only if $f$ is an fpqc epimorphism. For if $T \to L$ is any point, $K \times_L T$ is an fpqc $\ker(f)$-torsor. Since $\ker(f)$ is finite locally free, in particular of finite presentation,  $K \times_L T \to T$ is of finite presentation also. Thus for the isogenies and resolutions considered in our paper, there is no difference between surjectivity in fppf and fpqc topology.
\end{remark}

Zink \cite{Zin84} proves many useful results on formal Lie groups and isogenies, but in \cite{Zin84}, isogenies are defined to be morphisms of formal groups of the same dimension with kernel represented by an augmented nilpotent $R$-algebra, i.e. an augmented $R$-algebra with nilpotent augmentation ideal.  To apply the results of Zink in our setting, we begin by proving the equivalence of Zink's definition with ours.  Let us temporarily refer to Zink's notion of isogeny as a \textit{Zink isogeny} in what follows. It is likely that these results are known to the experts, but we provide full proofs, with references to the literature when appropriate.

Before proving the equivalence, we establish two preliminary lemmas. First, we show that Zink isogenies are finite locally free morphisms, which implies that their kernels are finite locally free group schemes. We then show that isogenous formal Lie groups have the same dimension, which is part of the definition of Zink isogenies.

\begin{lemma} \label{lemma: Zink implies projective}
If $f: K \to L$ is a Zink isogeny over $R$, then $\ca{O}(K)$ is a finite locally free $\ca{O}(L)$-module and the morphism $\ca{O}(L) \to \ca{O}(K)$ is faithfully flat.
\end{lemma}

\begin{proof}
Replacing $\Spec R$ by a Zariski open affine subset on which $K$ and $L$ admit coordinates, we may assume that $\ca{O}(K) = R[[x_1, \dots, x_d]]$ and $\ca{O}(L) = R[[y_1, \dots, y_d]]$.  Let $f$ map $y_i$ to $f_i \in \ca{O}(K)$, and set $I = (f_1, \dots, f_d)$.  Then $\ker(f)$ is represented by the spectrum of
\be
\ca{O}(K) \widehat{\ten}_{\ca{O}(L)} R = (\ca{O}(K)/I)^{\wedge}_{\mf{m}_K} = \ca{O}(K)/\ol{I},
\ee
where $\ol{I}$ is the closure of $I$ in $\ca{O}(K)$.  We claim that $I$ is in fact a closed ideal in $\ca{O}(K)$.  By assumption, $\ca{O}(K)/\ol{I}$ is a nilpotent $R$-algebra, so there is an $N>0$ such that $\mf{m}_K^N \subset \ol{I}$.  We will show that in fact $\mf{m}_K^N \subset I$; it follows from this that $I$ is an open ideal of $\ca{O}(K)$, and we then obtain the claim since open subgroups of topological groups are closed.

Let $M$ be the image of $\mf{m}_K^N$ in the quotient $\ca{O}(K)/I$. Since $\mf{m}_K^N \subset \ol{I}$, we have $\mf{m}_K^N \subset I + \mf{m}_K^{N+1}$, so $M \subset \mf{m}_K M$. By Nakayama's lemma we conclude $M = 0$, so that $\mf{m}_K^N \subset I$, as desired.

It follows that $\ker(f)$ is represented by the spectrum of $\ca{O}(K)/I$.  By nilpotence, we can apply the Weierstrass preparation theorem \cite[Theorem 5.8]{Zin84} to conclude that $f_*\ca{O}(K)$ is a finite locally free $\ca{O}(L)$-module and $\ca{O}(L) \to \ca{O}(K)$ is faithfully flat.
\end{proof}

\begin{lemma} \label{lemma: isogeny implies same dimension}
Let $R$ be a ring, and $f: K \to L$ an isogeny of formal Lie groups over $R$.  Then $\dim(K) = \dim(L)$. 
\end{lemma}

\begin{proof}
Since the dimension of a formal Lie group is determined fiberwise and is locally constant on the base, it suffices to prove the statement after base change to every residue field.  Thus we may assume that $R = k$ is a field.

Let $I$ be the augmentation ideal of $\mathcal{O}(L)$, let $L_r = \Spec \mathcal{O}(L)/I^r$, and let $K_r = L_r \times_L K$. Since $\ker(f)$ is representable by a finite group $k$-scheme and $K_r \to L_r$ is an fpqc $\ker(f)$-torsor, $K_r$ is representable by a finite $k$-scheme. Since filtered colimits commute with finite limits in the category of sets, $\colim_r K_r = (\colim_r L_r) \times_L K = K$. The maps $K_r \to K_{r+1}$ are closed immersions. Thus $\mathcal{O}(K) = \ilim_r \mathcal{O}(K_r)$ as a topological ring. If $J=  \mathcal{O}(K)I$, then $\mathcal{O}(K_r) = \mathcal{O}(K)/J^r$, so $J$ is an ideal of definition of $\mathcal{O}(K)$, so $J$ is $\mathfrak{m}_K$-primary where $\mathfrak{m}_K$ is the maximal ideal of $\mathcal{O}(K)$. 
Now 
\[ \dim_k \mathcal{O}(L)/I^r = \binom{r + \dim(L) -1}{\dim(L)}\]
so 
\[ \dim_k \mathcal{O}(K)/J^r = |\ker(f)|\binom{r+\dim(L)-1}{\dim(L)}.\]
Since $J$ is $\mathfrak{m}_K$-primary, dimension theory implies $\dim(K) = \dim(L)$, as desired.
\end{proof}

\begin{proposition} \label{proposition: Zink isogenies and isogenies}
Let $R$ be a ring, and $f: K \to L$ a morphism of formal Lie groups over $R$.  Then $f$ is an isogeny if and only if $f$ is a Zink isogeny.
\end{proposition}

\begin{proof}
First, suppose that $f$ is a Zink isogeny. By the proof of Lemma \ref{lemma: Zink implies projective}, $\ker(f)$ is represented by the spectrum of the uncompleted tensor product $\ca{O}(K) \ten_{\ca{O}(L)} R$, which is a finite projective $R$-module by Lemma \ref{lemma: Zink implies projective}.  Hence $\ker(f)$ is a finite locally free $R$-group scheme.

It remains to show that every section of $L$ lifts fpqc-locally to a section of $K$. Let $A$ be an $R$-algebra, and $y \in L(A) = \Hom_{\cts}(\ca{O}(L), A)$. Form $A' = A \ten_{\ca{O}(L)} \ca{O}(K)$.  Then $A'$ is fpqc over $A$ by Lemma \ref{lemma: Zink implies projective}, and by construction the map $y|_{A'}: \ca{O}(L) \to A \overset{\id \ten 1}\to A'$ factors as the composition of $f^*: \ca{O}(L) \to \ca{O}(K)$ and $x = 1 \ten \id: \ca{O}(K) \to A'$. It only remains to verify that $x \in K(A')$, i.e. that the map $x: \ca{O}(K) \to A'$ is continuous. Since $f$ is an isogeny, we have that $\mf{m}_K^n \subset \mf{m}_L \ca{O}(K)$ for some $n$ by \cite[Lemma 5.5]{Zin84}. Applying $x$, we find that $x(\mf{m}_K^n) \subset y(\mf{m}_L) A'$.  Since $y$ is continuous, there is an $N > 0$ such that $y(\mf{m}_L)^NA' = 0$; it follows that $x(\mf{m}_K^n)^{N} = 0$, so $x(\mf{m}_K)$ is nilpotent in $A'$ and the map $x: \ca{O}(K) \to A'$ is continuous, as desired.

Conversely, suppose that $f$ is an isogeny, and let $G = \ker(f)$. The augmentation ideal of $\ca{O}(K)$ is topologically nilpotent. Since there is a continuous surjection $\ca{O}(K) \to \ca{O}(G)$, the augmentation ideal of $\mathcal{O}(G)$ is nilpotent. 
Moreover $K$ and $L$ have the same dimension by Lemma \ref{lemma: isogeny implies same dimension}, so $f$ is a Zink isogeny.
\end{proof}

As a consequence, we have the following result, which one may also prove directly using the fact that isogenies are torsors under finite locally free group schemes and fpqc descent of finite locally free morphisms \cite[\href{https://stacks.math.columbia.edu/tag/02VO}{02VO}]{stacks-project}.

\begin{lemma} \label{lemma: isogeny flf coordinate rings}
If $f: K \to L$ is an isogeny, then $\ca{O}(K)$ is a finite locally free $\ca{O}(L)$-module.
\end{lemma}

\begin{proof}
This follows from Lemma \ref{lemma: Zink implies projective} and Proposition \ref{proposition: Zink isogenies and isogenies}.
\end{proof}

We now give an exposition of the Cartier-theoretic description of isogenies of formal Lie groups, following \cite{Zin84}. Throughout, let $R$ be a $p$-nilpotent ring and $\ol{R} = R/pR$.  We use an overline to denote the reduction to $\ol{R}$: of a map, formal Lie group, Cartier module, etc. 

Isogenies of formal Lie groups over $\ol{R}$ become isomorphisms after formally inverting Frobenius: a morphism $f: K \to L$ of formal Lie groups of the same dimension is an isogeny if and only if there is an $m \geq 0$ and a morphism $g: L \to K^{(p^m)}$ such that $g \circ f = F^m: K \to K^{(p^m)}$ \cite[Theorem 5.25]{Zin84}.  This leads to the definition of the \textit{$V$-divided Cartier module}.

Recall that if $N$ is a left $\E_{\ol{R}}$-module, then the Frobenius twist of $N$ is $N^{(p)} := \E_{\ol{R}} \ten_{\sigma, \E_{\ol{R}}} N$, where $\sigma: \E_{\ol{R}} \to \E_{\ol{R}}$ extends the Witt vector Frobenius via $\sigma(F) = F, \sigma(V) = V$.  There is a Frobenius map $F_{N}: N \to N^{(p)}$, defined by $F_{N}(n) = V \ten n$.  When $N = M(K)$ is the Cartier module of a commutative formal Lie group $K/{\ol{R}}$, then $N^{(p)} = M(K^{(p)})$ and $F_{N}$ is the morphism induced by the Frobenius map $K \to K^{(p)}$.  

\begin{definition}
Let $N$ be a left $\E_{\ol{R}}$-module. The $V$-divided Cartier module of $N$ is the left $\E_{\ol{R}}$-module
\be
\Div(N) := \colim N^{(p^i)}
\ee
where the colimit is taken along the Frobenius maps $F_{N^{(p^i)}}: N^{(p^i)} \to N^{(p^{i+1})}$.
\end{definition}
The functor $\Div$ is the Cartier-theoretic analogue of inverting the Frobenius, and consequently provides a natural framework for studying formal Lie groups up to isogeny. The following result is implicit in \cite[Chapter 5]{Zin84}, which explicitly handles the case of $\F_p$-algebras.

\begin{proposition}\label{proposition: isogeny and div}
Let $R$ be a $p$-nilpotent ring and $f: K \to L$ a morphism of commutative formal Lie groups over $R$. Then $f$ is an isogeny if and only if the induced map $\Div(M(\ol{f})): \Div(M(\ol{K})) \to \Div(M(\ol{L}))$ is an isomorphism.
\end{proposition}

\begin{proof}
By \cite[Remark 5.6 and Corollary 5.12]{Zin84}, $f$ is an isogeny if and only if $\ol{f}: \ol{K} \to \ol{L}$ is an isogeny over $\ol{R}$.  By \cite[Theorem 5.25]{Zin84}, if $\ol{f}$ is an isogeny then $\Div(M(\ol{f}))$ is an isomorphism. Conversely, if $\Div(M(\ol{f}))$ is an isomorphism, then by the last statement of \cite[Theorem 5.26]{Zin84}, $\ol{f}$ is an isogeny.
\end{proof}

As above, let $R$ be a $p$-nilpotent ring, $\ol{R} = R/pR$, and use an overline to denote reduction from $R$ to $\ol{R}$. The isogeny criterion of Proposition \ref{proposition: isogeny and div} is formulated for Cartier modules. In later applications, our Cartier--Dieudonné functor takes values in $D(\E_R)$, so we need a version of the criterion using the left derived functor $L\Div$ of the $V$-divided Cartier module functor $\Div$. This requires the following vanishing and acyclicity statements for Cartier modules of formal Lie groups.

Given a ring homomorphism $\phi: R \to S$, let $\phi: \E_R \to \E_S$ denote the induced map, which sends $F \mapsto F, V \mapsto V$, and $[x] \mapsto [\phi(x)]$. Regard $\E_S$ as an $(\E_S, \E_R)$-bimodule with the usual left $\E_S$-action and with right $\E_R$-action $x \cdot a = x \phi(a)$.

\begin{lemma} \label{lemma: tor vanishing}
Let $M$ be a Cartier module of a formal Lie group over $R$, and let $\phi: R \to S$ be a map of $p$-nilpotent rings.  Then $\Tor_i^{\E_R}(\E_S, M) = 0$ for all $i > 0$.
\end{lemma}

\begin{proof}
Suppose first that $M/VM$ is a free $R$-module.
By Proposition \ref{prop: structure equations}, we may present $M$ by structural equations
\be
0 \to \E_R^r \overset{u}\to \E_R^r \to M \to 0,
\ee
where on a basis $e_1, \dots, e_r$ we have $u(e_i) = Fe_i - \sum_{j = 1}^r \sum_{n \geq 0} V^n[c_{i,n,j}]e_j$ for some $c_{i,n,j} \in R$.  It follows that $\Tor_i^{\E_R}(N, M) = 0$ for any right $\E_R$-module $N$ and all $i > 1$.  We also have
\be
\Tor_1^{\E_R}(\E_S, M) = \ker(\E_S \ten_{\E_R} \E_R^r \overset{1 \ten u}\to \E_S \ten_{\E_R} \E_R^r).
\ee
We have an identification of left $\E_S$-modules $\E_S \ten_{\E_R} \E_R^r \cong \E_S^r$ via $x \ten y \mapsto x \phi(y)$. Under this, the map $1 \ten u$ corresponds to the map $u^{(\phi)}: \E_S^r \to \E_S^r$, $u^{(\phi)}(e_i) = Fe_i - \sum_{j,n} V^n [\phi(c_{i,n,j})]e_j$.  This is a structural equation, so by the injectivity of structural equations (Proposition \ref{prop: structure equations}), $\Tor_1^{\E_R}(\E_S, M) = 0$.

Now let $M$ be any Cartier module of a formal Lie group over $R$.  Since $M/VM$ is a finite projective $R$-module, we may choose a finite projective $R$-module $Q$ such that $Q \oplus M/VM$ is a free $R$-module of finite rank.  Consider the additive formal vector group $\widehat{\bb{V}}(Q)$ on $Q$, i.e. the formal completion at the zero section of $\Spec(\Sym_R^\bu(Q^\vee))$. This is a finite-dimensional formal Lie group with $\Lie(\widehat{\bb{V}}(Q)) \cong Q$; it follows that the Cartier module $M_Q$ of $\widehat{\bb{V}}(Q)$ satisfies $M_Q/VM_Q \cong Q$.  If we set $M' := M \oplus M_Q$, then $M'$ is a Cartier module of a formal Lie group over $R$ such that $M'/VM'$ is free.  By what we have shown previously we conclude
\be
\Tor_i^{\E_R}(\E_S, M') = 0
\ee
for all $i > 0$.  Since $\Tor_i$ preserves direct sums, we find that $\Tor_i^{\E_R}(\E_S, M) = 0$, as desired.
\end{proof}

Note that the $V$-divided Cartier module functor $\Div$ is right exact, and therefore admits a left derived functor $L\Div: D(\E_{\ol{R}}) \to D(\E_{\ol{R}})$, given by
\be
L\Div(X) = \hocolim_{n} \left(\E_{\ol{R}}^{(p^n)} \Lotimes{\E_{\ol{R}}} X \right),
\ee
where $\E_{\ol{R}}^{(p^n)}$ is the $\E_{\ol{R}}$-bimodule whose underlying left $\E_{\ol{R}}$-module is $\E_{\ol{R}}$ and with right action given by the map $\E_{\ol{R}} \to \E_{\ol{R}}$ induced by the $n$-th iterate of the Frobenius map of $\ol{R}$. 

\begin{corollary} \label{corollary: higher div}
Let $M$ be a Cartier module of a formal Lie group over $\ol{R}$, viewed in the derived category $D(\E_{\ol{R}})$.  Then $L\Div(M) = \Div(M)$.
\end{corollary}

\begin{proof}
By definition, $L\Div(M)$ is the homotopy colimit of $\E_{\ol{R}}^{(p^n)} \Lotimes{\E_{\ol{R}}} M$.  Since filtered colimits are exact,
\be
H^{-i}(L\Div(M)) = \colim_n \Tor_i^{\E_{\ol{R}}}(\E_{\ol{R}}^{(p^n)}, M).
\ee
Lemma \ref{lemma: tor vanishing} implies that $\Tor_i^{\E_{\ol{R}}}(\E_{\ol{R}}^{(p^n)}, M) = 0$ for $i > 0$. It follows that $H^{-i}(L\Div(M)) = 0$ for $i \neq 0$, and consequently $L\Div(M) = \Div(M)$.
\end{proof}

This leads to the desired derived characterization of isogenies of formal Lie groups, which yields one of the three conditions in the description of the essential image of the Cartier--Dieudonn\'e functor.

\begin{proposition} \label{proposition: derived isogenies}
Let $f: K \to L$ be a morphism of commutative formal Lie groups over $R$, let $F$ be the fiber in $D(\E_R)$ of the map $M(K) \to M(L)$ induced by $f$, and let $\ol{F} := \E_{\ol{R}} \Lotimes{\E_R} F$. Then $f$ is an isogeny if and only if $L\Div(\ol{F}) = 0$.
\end{proposition}

\begin{proof}
Base change to $\E_{\ol{R}}$ yields a fiber sequence
\be
\ol{F} \to \E_{\ol{R}} \Lotimes{\E_R} M(K) \to \E_{\ol{R}} \Lotimes{\E_R} M(L) \overset{+1}\to.
\ee
By Lemma \ref{lemma: tor vanishing}, we have $\E_{\ol{R}} \Lotimes{\E_R} M(K) \cong \E_{\ol{R}} \ten_{\E_R} M(K)$, so by compatibility of Cartier theory with base change \cite[Theorems 4.43 and 4.45]{Zin84} we have $\E_{\ol{R}} \ten_{\E_R} M(K) \cong M(\ol{K})$.  As the same reasoning applies to $L$, it follows that $\ol{F} \cong \Fib(M(\ol{K}) \to M(\ol{L})) \in D(\E_{\ol{R}})$. Since $L\Div$ is a triangulated functor, applying $L\Div$ yields a distinguished triangle 
\be
L\Div(\ol{F}) \to L\Div(M(\ol{K})) \to L\Div(M(\ol{L})) \to L\Div(\ol{F})[1].
\ee
By Corollary \ref{corollary: higher div}, the Cartier modules $M(\ol{K})$ and $M(\ol{L})$ are $L\Div$-acyclic.  Therefore, this triangle yields an identification of $L\Div(\ol{F})$ with the fiber $\Fib(\Div(M(\ol{K})) \to \Div(M(\ol{L})))$ in $D(\E_{\ol{R}})$, so $L\Div(\ol{F}) = 0$ if and only if $\Div(M(\ol{K})) \to \Div(M(\ol{L}))$ is an isomorphism.  By Proposition \ref{proposition: isogeny and div}, this happens if and only if $f$ is an isogeny.
\end{proof}

\section{Classification of Connected Group Schemes} \label{section: connected theorem}
Fix a $p$-nilpotent ring $R$. In this section, we begin by characterizing fiberwise-connected group schemes over $R$ in a few equivalent ways.  We then define our Cartier--Dieudonn\'e functor and prove it is an equivalence.

\subsection{Connected Group Schemes and Formal Lie Groups}
Let us recall the definition of the relevant category of group schemes from the introduction.

\begin{definition}
Let $C(R)$ denote the category of finite locally free commutative $R$-group schemes $G$ which are \textit{fiberwise connected}, i.e. such that $G_s$ is a connected $k(s)$-group scheme for all $s \in \Spec R$.
\end{definition}
Note that $G \in C(R)$ implies that the order of $G$ is, locally on $\Spec(R)$, a power of $p$.  We sometimes refer to such group schemes as ``connected" in what follows.  We begin by providing equivalent algebraic and geometric ways to view connected group schemes.  Recall that an \textit{augmented nilpotent $R$-algebra} is an augmented $R$-algebra with nilpotent augmentation ideal.

\begin{proposition} \label{proposition: connected conditions}
Suppose $G = \Spec(A)$ is a finite locally free commutative $R$-group scheme. The following are equivalent:
\ben
\item $G \in C(R)$; \label{G in C(R)}
\item $A$ is an augmented nilpotent $R$-algebra;\label{A nil aug}
\item $G$ is a closed subgroup of a commutative formal Lie group over $R$.
\label{G subgroup of Lie}
\een
\end{proposition}

\begin{proof}
\ref{G in C(R)}$\implies$\ref{A nil aug}: Suppose $G \in C(R)$. Let $I \subset A$ be the augmentation ideal, and let $s \in \Spec(R)$. Since the short exact sequence of $R$-modules $0 \to I \to A \to R \to 0$ splits, it remains exact after base change to the residue field $k(s)$ of $s$, and therefore the augmentation ideal of $G \times_R k(s)$ is $I_s := I \ten_R k(s)$.  Since finite connected group schemes over a field are spectra of Artinian local rings \cite[Tag \href{https://stacks.math.columbia.edu/tag/0BUG}{0BUG}]{stacks-project}, we have that $I_s$ is nilpotent in $A_s := A \ten_R k(s)$.  It follows that $I$ is contained in the nilradical of $A$.  Indeed, if $\mf{p} \subset A$ is a prime of $A$ above $s \in \Spec(R)$, the image of $I$ in $A_s$ is contained in every prime of $A_s$ and thus $I \subset \mf{p}$. Since $G$ is finite, $I$ is finitely generated, so $I$ is nilpotent. 

\ref{A nil aug}$\implies$\ref{G in C(R)}: Suppose that $A$ is an augmented nilpotent $R$-algebra.  Base changing to any residue field $k(s)$, the augmentation ideal becomes nilpotent and the quotient is $k(s)$, which implies that $A_s$ is an Artinian local ring with maximal ideal $I_s$, so the fiber $G_s$ is connected.

\ref{G in C(R)}$\implies$\ref{G subgroup of Lie}: By Cartier duality, we may identify $G$ with $\ul{\Hom}_R(G^\vee, \G_m)$, where $G^\vee$ is the Cartier dual of $G$. Following a construction of B\'egueri \cite[Proposition 2.2.1]{Beg80}, there is a canonical embedding of $G$ into the abelian sheaf $H := \ul{\Map}_R(G^\vee, \G_m)$ of algebraic maps from the Cartier dual $G^\vee$ of $G$ to $\G_m$.  Moreover, $H$ is canonically isomorphic to the Weil restriction $\Res_{G^\vee/R}(\G_{m, G^\vee})$ and is therefore a smooth affine $R$-group scheme \cite[Proposition 4.4]{Sch94}.  Since $G$ is the spectrum of an augmented nilpotent $R$-algebra, the closed embedding $G \to H$ factors through an infinitesimal neighborhood of the identity of $H$, and therefore through the formal completion $\widehat{H}_e$ of $H$ at the identity section, providing a closed embedding of $G$ into a commutative formal Lie group over $R$.

\ref{G subgroup of Lie}$\implies$\ref{G in C(R)}: Let $G \to K$ be a closed embedding of groups into a formal Lie group $K$. After base changing to $k(s)$, we get a closed subgroup $G_s \to K_s$ over the field $k(s)$.  On coordinate rings, this corresponds to a continuous surjection $k(s)[[x_1, \dots, x_d]] \to A_s$ which takes the augmentation ideal $(x_1, \dots, x_d)$ onto the augmentation ideal of $A_s$.  Moreover, continuity implies that some power of the augmentation ideal $(x_1, \dots, x_d)$ maps to zero in $A_s$. Hence the augmentation ideal of $A_s$ is nilpotent, so $A_s$ is Artinian local and thus $G \in C(R)$.
\end{proof}

\begin{remark} B\'egueri \cite[Proposition 2.2.1]{Beg80} considers all finite locally free commutative group schemes over arbitrary base schemes, and shows in addition that the fppf cokernel of the map $G \to \Res_{G^\vee/R}(\G_{m, G^\vee})$ may be identified with the smooth affine group scheme of symmetric 2-cocycles on $G^\vee$ with values in $\G_m$. 
\end{remark}

We now show that commutative connected group schemes over $p$-nilpotent bases admit functorial 2-term resolutions by commutative formal Lie groups (Corollary \ref{corollary: begueri for connected}).  Our result follows from the general fact that quotients of formal Lie groups by finite locally free subgroups are again formal Lie groups (Theorem \ref{theorem: quotients of FLG}). First, we establish a couple preliminary lemmas.

\begin{lemma}\label{lem: quotient sheaf is ind-scheme}
    Let $R$ be a ring, $K \to \Spec R$ a formal Lie group, $G \subset K$ a finite locally free closed normal group scheme.
    Then the fppf quotient sheaf $K/G$ is representable by an ind-infinitesimal $\alepho$-ind-scheme over $R$.
\end{lemma}
\begin{proof}
    Suppose that $K = \Spf(A)$ with augmentation ideal $\mathfrak{a} \subset A$. Let $I \subset A$ be the ideal of $G \subset K$, so that $A/I$ is a finite locally free $R$-module. Since the map $A \to A/I$ is continuous, this map factors through $A\to A/\mathfrak{a}^N$ for some $N$, that is, $\mathfrak{a}^N \subseteq I \subseteq \mathfrak{a}$. Hence the $\mathfrak{a}$-adic and $I$-adic topologies on $A$ agree. Set $K_n = \Spec(A/I^{n+1})$ for all $n \geq 0$; we have $K = \colim_n K_n$.

    Since $G$ is a subgroup of $K$, right translation by $G$ preserves $I$, so the translation action of $G$ on $K$ restricts to a free action of $G$ on $K_n$ for all $n$.
    Since $K_n$ is affine and $G$ is finite locally free, the quotient $K_n/G$ is representable by an $R$-scheme $Q_n$, and the map $K_n \to Q_n$ is an fppf $G$-torsor \cite[Tag \href{https://stacks.math.columbia.edu/tag/07S7}{07S7}]{stacks-project}. It follows by \cite[Tag \href{https://stacks.math.columbia.edu/tag/01ZT}{01ZT}]{stacks-project} that $Q_n$ is affine.
    For any $n \leq m$, the $G$-equivariant map $K_n \to K_m \times_{Q_m} Q_n$ is a morphism of $G$-torsors over $Q_n$, and therefore an isomorphism; this shows that the inclusions and projection morphisms fit into a Cartesian square
\be
\begin{tikzcd}
K_n \arrow[r] \arrow[d] & K_m \arrow[d] \\
Q_n \arrow[r] & Q_m.
\end{tikzcd}
\ee
It follows by faithfully flat descent that the natural maps $Q_n \to Q_m$ are closed immersions \cite[Tag \href{https://stacks.math.columbia.edu/tag/02L6}{02L6}]{stacks-project}. Define $Q := \colim_n Q_n$ in the category of fppf sheaves.  Since $Q_n \to Q_m$ is a closed immersion for $m \geq n$, the canonical maps $Q_n \to Q$ are all monomorphisms.

We claim that $Q$ is isomorphic to the fppf quotient sheaf $K/G$, which we show by checking the universal property.  For any fppf sheaf $F$, we have $\Hom(Q, F) = \Hom(\colim_n Q_n, F) = \ilim_n \Hom(Q_n, F)$.  Note that $Q_n = K_n/G$, so a map $Q_n \to F$ is the same as a $G$-invariant map $K_n \to F$, where $F$ is given the trivial $G$-action, so $\Hom(Q, F) = \ilim_n \Hom_G(K_n, F)$. But a map $K \to F$ is $G$-invariant if and only if the induced maps $K_n \to F$ are, and thus $\ilim_n \Hom_G(K_n, F) = \Hom_G(K, F)$, showing that $Q \cong K/G$.

    Since $(K_n,e)$ is infinitesimal, the ideal of $G \subseteq K_n$ is nilpotent, so descent implies $(Q_n,e)$ is infinitesimal. Thus $Q = \colim_n Q_n$ is an ind-infinitesimal $\alepho$-ind-scheme over $R$.
\end{proof}

\begin{lemma}\label{lem: quotient quasiregular immersion}
    Hypotheses as in Lemma \ref{lem: quotient sheaf is ind-scheme}. Let $A$ be the topological coordinate ring of $K$ and $I \subseteq A$ the ideal of the closed immersion $G \to K$.
    Then $I/I^2$ is a finite locally free $A/I$-module and the natural map
    \begin{equation*}
        \Sym_{A/I}(I/I^2) \to \bigoplus_{r \geq 0} I^r/I^{r+1}
    \end{equation*}
    is an isomorphism of graded $A/I$-algebras.
\end{lemma}

\begin{proof}
The statement is Zariski-local on $\Spec R$, so passing to an open subset we may assume that $A = R[[x_1, \dots, x_d]]$ with augmentation ideal $(x_1, \dots, x_d)$, where $d$ is the dimension of $K$ on the open subset.
Write $B = A/I$. Let $P = R[x_1, \dots, x_d]$ and $\mf{p} = (x_1, \dots, x_d) \subset P$. The surjection $A \to B$ induces a surjection $P \to B$; let $\mf{q} = \ker(P \to B)$. We have $\mf{p}^N \subset \mf{q} \subset \mf{p}$ for some $N$, so the $\mf{p}$-adic and $\mf{q}$-adic filtrations are cofinal and consequently $A$ is also the completion of $P$ with respect to $\mf{q}$. 

The surjection $P \to B$ corresponds to a closed immersion $G \to \A^d_R$. Moreover, by \cite[Chapter II, Lemma 3.2.5]{Mes72} and \cite[Exp. III, Proposition 4.15]{SGA3I} finite locally free group schemes are local complete intersections. It follows that the morphism $G \to \A^d_R$ is a Koszul-regular immersion \cite[Tag \href{https://stacks.math.columbia.edu/tag/069G}{069G}]{stacks-project}, hence a quasi-regular immersion \cite[Tag \href{https://stacks.math.columbia.edu/tag/063K}{063K}]{stacks-project}. By \cite[Tag \href{https://stacks.math.columbia.edu/tag/063H}{063H}]{stacks-project} $\mf{q}$ is a finitely generated ideal of $P$, $\mf{q}/\mf{q}^2$ is a finite locally free $B$-module and the map 
\be
\Sym_B^\bu(\mf{q}/\mf{q}^2) \to \bigoplus_{r \geq 0} \mf{q}^r/\mf{q}^{r+1}
\ee
is an isomorphism.

We now show that there are compatible isomorphisms of $B$-modules $I^r/I^{r+1} \cong \mf{q}^r/\mf{q}^{r+1}$ for all $r \geq 0$, from which the lemma follows. Since $A = P^\wedge_{\mf{q}}$ and $\mf{q}$ is finitely generated, we have by \cite[Tag \href{https://stacks.math.columbia.edu/tag/05GG}{05GG}]{stacks-project} that $A/\mf{q}^rA = P/\mf{q}^r$ for all $r \geq 1$, which upon taking $r = 1$ implies $I = \mf{q}A$.  Consequently, we may identify the kernel of $A/I^{r+1} \to A/I^r$ with the kernel of $P/\mf{q}^{r+1} \to P/\mf{q}^r$ to conclude $I^r/I^{r+1} \cong \mf{q}^r/\mf{q}^{r+1}$ for all $r \geq 0$. These isomorphisms are all induced by $P \to A = P^{\wedge}_{\mf{q}}$ and therefore they preserve multiplication.
\end{proof}

\begin{theorem} \label{theorem: quotients of FLG}
    Let $R$ be a ring, $K \to \Spec R$ a formal Lie group, $G \subset K$ a finite locally free closed normal group scheme.
    Then $K/G$ is a formal Lie group over $R$.
\end{theorem}
\begin{proof}
    By Lemma \ref{lem: quotient sheaf is ind-scheme}, $Q=K/G$ is an $\alepho$-ind-scheme over $R$.
    Let $C$ be the topological coordinate ring of $Q$.
    Since $(Q,e)$ is ind-infinitesimal, $C$ is separated and complete with respect to the ideal $J$ of $e$.
    To show that $Q$ is a formal Lie group over $R$, we must show that $C$ is isomorphic to a completed symmetric algebra of a finite locally free $R$-module.
    Let $A$ be the topological coordinate ring of $K$ and $I \subseteq A$ the ideal of $G$.
    
    Consider the associated graded $\bigoplus_n J^n /J^{n+1}$ of $C$ with respect to $J$.
    Since $G$ is the kernel of $K \to Q$, the pullback of $J$ along $K \to Q$ is $I$. 
    Thus, the pullback along $G \to \Spec R$ of $J^n/J^{n+1}$ is $I^n/I^{n+1}$.
    By Lemma \ref{lem: quotient quasiregular immersion}, $I/I^2$ is a finite locally free $A/I$-module and the natural map of graded $A/I$-algebras
    \[ \Sym_{A/I}(I/I^2) \to \bigoplus_n I^n/I^{n+1}\]
    is an isomorphism.
    Since $G\to \Spec R$ is faithfully flat, $J/J^2$ is a finite locally free $R$-module and the natural map of graded $R$-algebras
    \[ \Sym_R(J/J^2) \to \bigoplus_n J^n/J^{n+1}\]
    is an isomorphism.

    Now we show that $C$ is isomorphic to a completed symmetric algebra over $R$.
    Since $J/J^2$ is finite locally free over $R$, there exists an $R$-linear splitting $\psi: J/J^2 \to J$ of the quotient map $J \to J/J^2$.
    Then $\psi$ defines an $R$-algebra homomorphism
    $\phi: \Sym_R(J/J^2) \to C$
    which induces the canonical map $\Sym_R(J/J^2) \to \bigoplus_{n} J^n/J^{n+1}$ on associated graded.
    Since $C$ is separated and complete with respect to $J$ and $\phi$ is an isomorphism on associated graded, the map $\phi$ induces an isomorphism $\hat{\phi}: \widehat{\Sym}_R(J/J^2) \overset{\sim}{\to} C$.
    Thus $\Spf(C)$ is a polydisk and $Q$ is a formal Lie group.
\end{proof}

\begin{corollary}[The B\'egueri Resolution for Connected Group Schemes] \label{corollary: begueri for connected}
Suppose $G \in C(R)$.  Then there is a functorial short exact sequence $0 \to G \to K \to L \to 0$ with $K$ and $L$ commutative formal Lie groups over $R$.
\end{corollary}

\begin{proof}
We have already seen that there is a functorial embedding $G \to K$ with $K$ a commutative formal Lie group over $R$, so the result follows by taking $L$ to be the fppf quotient $K/G$ which is a formal Lie group by Theorem \ref{theorem: quotients of FLG}.
\end{proof}

\subsection{The Cartier--Dieudonné functor}
Throughout, let $R$ be a $p$-nilpotent ring. 

\begin{definition}
    Let $\uM: C(R) \to D(\E_R)$ be the functor 
    \[ \uM(G) := \tau^{\leq 1} R\Hom_{\fpqc}(\What,G).\]
\end{definition}

Here $R\Hom_{\fpqc}$ is computed in the category of fpqc sheaves of abelian groups over $R$, and the resulting object carries a natural $\E_R$-module structure via the identification $\E_R = \End(\What)^{\op}$. We begin by explaining how to compute $\uM(G)$ via two-term resolutions by formal Lie groups. 

\begin{proposition}\label{prop: resolution of MG}
Let $G \in C(R)$, and suppose $0 \to G \to K \to L \to 0$ is a short exact sequence where $K$ and $L$ are commutative formal Lie groups over $R$.  Then there is an isomorphism 
\be
\uM(G) \cong \Fib(M(K) \to M(L))
\ee
in $D(\E_R)$.
\end{proposition}

\begin{proof}
We have a distinguished triangle
\begin{equation} \label{equation: RHom triangle}
R\Hom_{\fpqc}(\What, G) \to R\Hom_{\fpqc}(\What, K) \to R\Hom_{\fpqc}(\What, L) \overset{+1}\to
\end{equation}
which induces a map
\begin{equation} \label{equation: truncation map}
\tau^{\leq 1} R\Hom_{\fpqc}(\What, G) \to \Fib(\tau^{\leq 1} R\Hom_{\fpqc}(\What, K) \to \tau^{\leq 1} R\Hom_{\fpqc}(\What, L)).
\end{equation}
By Proposition \ref{prop: ext1 from What}, $\Ext^1_{\fpqc}(\What, K) = \Ext^1_{\fpqc}(\What, L) = 0$, so $\tau^{\leq 1} R\Hom_{\fpqc}(\What, K) = \Hom(\What, K)$ and similarly for $L$.  The long exact sequence in cohomology associated to the distinguished triangle (\ref{equation: RHom triangle}) then shows that the induced map (\ref{equation: truncation map}) is a quasi-isomorphism
\be
\tau^{\leq 1} R\Hom_{\fpqc}(\What, G) \cong \Fib(\Hom(\What, K) \to \Hom(\What, L))
\ee
as desired.
\end{proof}

The Proposition shows that $\uM(G)$ can be computed via any two-term resolution of $G$ by formal Lie groups, and consequently that any such resolution produces a well-defined object of $D(\E_R)$.  We may also prove this fact directly:

\begin{proposition}
Suppose that $0 \to G \to K^0 \to K^1 \to 0$ and $0 \to G \to L^0 \to L^1 \to 0$ are two resolutions of $G$ by commutative formal Lie groups over $R$.  Then there is an isomorphism $\Fib(M(K^0) \to M(K^1)) \cong \Fib(M(L^0) \to M(L^1))$ in $D(\E_R)$.
\end{proposition}

\begin{proof}
Diagonally embed $G$ into $K^0 \times L^0$ and let $Q$ be the cokernel, a formal Lie group by Theorem \ref{theorem: quotients of FLG}. We have a commutative diagram
\be
\begin{tikzcd}
0 \arrow[r] & G \arrow[d] \arrow[r] & K^0 \times L^0 \arrow[r] \arrow[d] & Q \arrow[r] \arrow[d] & 0 \\
0 \arrow[r] & G \arrow[r] & K^0 \arrow[r] & K^1 \arrow[r] & 0.
\end{tikzcd}
\ee
Clearly $\ker(K^0 \times L^0 \to K^0) = L^0$, while $\ker(Q \to K^1) = (G \times L^0)/G \cong L^0$ and the induced map $L^0 \to L^0$ is the identity. By the exactness of the Cartier module functor on formal Lie groups (Corollary \ref{corollary: M is exact on FLGs}), we obtain a diagram of Cartier modules
\be
\begin{tikzcd}
0 \arrow[r] & M(L^0) \arrow[d] \arrow[r] & M(K^0 \times L^0) \arrow[r] \arrow[d] & M(K^0) \arrow[r] \arrow[d] & 0 \\
0 \arrow[r] & M(L^0) \arrow[r] & M(Q) \arrow[r] & M(K^1) \arrow[r] & 0.
\end{tikzcd}
\ee
The left vertical arrow is an isomorphism, so by the snake lemma, the map of complexes $[M(K^0 \times L^0) \to M(Q)] \to [M(K^0) \to M(K^1)]$ is a quasi-isomorphism.  Since we can make the same argument for $M(L^0) \to M(L^1)$, it follows that $\Fib(M(K^0) \to M(K^1))$ and $\Fib(M(L^0) \to M(L^1))$ are isomorphic in $D(\E_R)$, as desired.
\end{proof}

\begin{remark} \label{remark: alternate construction}
We have seen that $G \in C(R)$ admits a functorial 2-term resolution by formal Lie groups (Corollary \ref{corollary: begueri for connected}).  Combining this with Proposition \ref{prop: resolution of MG} therefore provides an alternate construction of the functor $\uM: C(R) \to D(\E_R)$.
\end{remark}

\subsection{Full faithfulness of the Cartier--Dieudonné functor}

\begin{theorem} \label{theorem: fully faithful}
    Let $R$ be a p-nilpotent ring. Then the Cartier--Dieudonn\'e functor
    \[ \uM: C(R) \to D(\E_R)\]
    is fully faithful. 
\end{theorem}
\begin{proof}
    First, we show that $M: \Hom(G,L) \to \Hom_{D(\E_R)}(\uM(G),M(L))$ is an isomorphism when $G \in C(R)$ and $L$ is a formal Lie group.
    Let $0 \to G \to K^0 \to K^1 \to 0$ be a short exact sequence in the fpqc topology, where $K^0$ and $K^1$ are formal Lie groups.
    By Proposition \ref{prop: resolution of MG},
    there is a distinguished triangle 
    \[ \uM(G) \to M(K^0) \to M(K^1) \overset{+1}{\to}\]
    in $D(\E_R)$.

    Consider the following diagram:
    \begin{equation}\label{eq: full faithfulness 5 lemma}
\begin{tikzcd}[ampersand replacement=\&,cramped]
	{\Hom(K^1,L)} \& {\Hom(M(K^1),M(L))} \\
	{\Hom(K^0,L)} \& {\Hom(M(K^0),M(L))} \\
	{\Hom(G,L)} \& {\Hom(\uM(G),M(L))} \\
	{\Ext^1_{\fpqc}(K^1,L)} \& {\Ext^1(M(K^1),M(L))} \\
	{\Ext^1_{\fpqc}(K^0,L)} \& {\Ext^1(M(K^0),M(L))}
	\arrow[from=1-1, to=1-2]
	\arrow[from=1-1, to=2-1]
	\arrow[from=1-2, to=2-2]
	\arrow[from=2-1, to=2-2]
	\arrow[from=2-1, to=3-1]
	\arrow[from=2-2, to=3-2]
	\arrow[from=3-1, to=3-2]
	\arrow[from=3-1, to=4-1]
	\arrow["{(\ast)}"{description}, draw=none, from=3-1, to=4-2]
	\arrow[from=3-2, to=4-2]
	\arrow["{\tilde M}", from=4-1, to=4-2]
	\arrow[from=4-1, to=5-1]
	\arrow[from=4-2, to=5-2]
	\arrow["{\tilde M}", from=5-1, to=5-2]
\end{tikzcd}
\end{equation}
where the columns are the long exact sequences obtained by applying $\Hom(-,L)$ and $\Hom(-,M(L))$ in the derived categories of fpqc abelian sheaves and $\E_R$-modules, respectively. We claim that the diagram commutes. 
All squares except $(\ast)$ follow from naturality.
For $(\ast)$: suppose that $f: G\to L$ is a homomorphism. Let $X$ be the pushout in the fpqc topology
\[
\begin{tikzcd}[cramped]
	0 & G & {K^0} & {K^1} & 0 \\
	0 & L & X & {K^1} & 0
	\arrow[from=1-1, to=1-2]
	\arrow[from=1-2, to=1-3]
	\arrow[from=1-2, to=2-2]
	\arrow[from=1-3, to=1-4]
	\arrow[from=1-3, to=2-3]
	\arrow[from=1-4, to=1-5]
	\arrow[from=1-4, to=2-4]
	\arrow[from=2-1, to=2-2]
	\arrow[from=2-2, to=2-3]
	\arrow["\lrcorner"{anchor=center, pos=0.125, rotate=180}, draw=none, from=2-3, to=1-2]
	\arrow[from=2-3, to=2-4]
	\arrow[from=2-4, to=2-5]
\end{tikzcd}.
\]
By Theorem \ref{theorem: alepho-Lie closed under extensions} and Corollary \ref{corollary: M is exact on FLGs}, $X$ is a formal Lie group and the sequence of Cartier modules $0 \to M(L) \to M(X) \to M(K^1) \to 0$ is exact. The composition $\Hom(G, L) \to \Ext^1_{\fpqc}(K^1, L) \to \Ext^1(M(K^1), M(L))$ sends $f$ to the class $[M(X)] \in \Ext^1_{\E_R}(M(K^1),M(L))$.
On the other hand, 
\[ 0 \to G \to K^0 \oplus L \to X \to 0\]
is exact in the fpqc topology, so Proposition \ref{prop: resolution of MG} shows that 
\[ \uM(G) \to M(K^0) \oplus M(L) \to M(X) \]
is a fiber sequence. Hence the composition $\Hom(G, L) \to \Hom(\uM(G), M(L)) \to \Ext^1(M(K^1), M(L))$ also sends $f$ to $[M(X)]$, showing that $(\ast)$ commutes.

By Cartier theory (Theorem \ref{theorem: cartier in countable type}) and Lemma \ref{lemma: cartier functor on ext1}, all horizontal arrows in \eqref{eq: full faithfulness 5 lemma} but the middle are isomorphisms. By the 5-lemma, the map
\[ \Hom(G,L) \to \Hom_{D(\E_R)}(\uM(G),M(L))\] 
is an isomorphism.

Now we prove the theorem. Let $0 \to H \to L^0 \to L^1 \to 0$ be a resolution of $H$ by formal Lie groups $L^0$ and $L^1$.
First, note that if $L$ is a formal Lie group, then
\[ \Ext^{-1}(\uM(G),M(L)) = 0.\]
For consider the long exact sequence
\[
\begin{tikzcd}[ampersand replacement=\&]
	\cdots \& {\Ext^{-1}(M(K^0),M(L))} \& {\Ext^{-1}(\uM(G),M(L))} \\
	{\Hom(M(K^1),M(L))} \& {\Hom(M(K^0),M(L))} \& \cdots
	\arrow[from=1-1, to=1-2]
	\arrow[from=1-2, to=1-3]
	\arrow[from=1-3, to=2-1]
	\arrow[from=2-1, to=2-2]
	\arrow[from=2-2, to=2-3]
\end{tikzcd}
\]
Since $M(K^0)$ and $M(L)$ are concentrated in degree 0, $\Ext^{-1}(M(K^0),M(L)) = 0$. Since $K^0 \to K^1$ is surjective in the fpqc topology and $M$ is fully faithful, $\Hom(M(K^1),M(L)) \to \Hom(M(K^0),M(L))$ is injective, so $\Ext^{-1}(\uM(G),M(L)) = 0$.
This being so, applying $R\Hom(\uM(G),-)$ to the triangle $\uM(H) \to M(L^0) \to M(L^1) \to^{+1}$ gives the diagram 
\[
\begin{tikzcd}[ampersand replacement=\&]
	0 \& {\Hom(G,H)} \& {\Hom(G,L^0)} \& {\Hom(G,L^1)} \\
	0 \& {\Hom(\uM(G),\uM(H))} \& {\Hom(\uM(G),M(L^0))} \& {\Hom(\uM(G),M(L^1))}
	\arrow[from=1-1, to=1-2]
	\arrow[from=1-1, to=2-1]
	\arrow[from=1-2, to=1-3]
	\arrow[from=1-2, to=2-2]
	\arrow[from=1-3, to=1-4]
	\arrow[from=1-3, to=2-3]
	\arrow[from=1-4, to=2-4]
	\arrow[from=2-1, to=2-2]
	\arrow[from=2-2, to=2-3]
	\arrow[from=2-3, to=2-4]
\end{tikzcd}
\]
The last two vertical arrows are isomorphisms, hence $\Hom(G,H) \to \Hom(\uM(G),\uM(H))$ is an isomorphism, as desired.
\end{proof}

\begin{example}\label{ex: oort-functor}
We may now exhibit why Oort's functor $M^{\Oort}$, assigning $G \in C(R)$ to $M^{\Oort}(G) = \coker(M(K^0) \to M(K^1))$ when $0 \to G \to K^0 \to K^1 \to 0$ is a resolution, is not fully faithful. By Proposition \ref{prop: resolution of MG}, $M^{\Oort}(G) = H^1(\uM(G))$.

    Let $R$ be an $\Fp$-algebra and $G = \alpha_p$.
    Since $\alpha_p$ has the resolution
    \[ 0 \to \alpha_p \to \Ghat_a \overset{F}{\to} \Ghat_a \to 0,\]
    the Cartier--Dieudonné module of $\alpha_p$ is 
    \[ \uM(\alpha_p) = [\E_R/\E_R F \overset{r_V}{\to} \E_R / \E_R  F],\]
    where $r_V$ is right multiplication by $V$. Thus, 
    \[M^{\Oort}(\alpha_p) = \E_R/(\E_R F + \E_R V).\]
    Using that every element of $\E_R$ has a unique representation of the form $\sum_{n,m}V^n[c_{n,m}]F^m$, it can be shown that $\End_{\E}(\E_R/(\E_R F + \E_R V))\cong R^p$ where $c^p \in R^p$ acts by $x \mapsto x[c^p]$.
    The map 
    \[ R= \End(\alpha_p) \to \End_{\E_R}(M^{\Oort}(\alpha_p))\]
    sends $c \in R$ to $c^p$, as $c: \alpha_p \to \alpha_p$ fits into the following diagram:
    
\[
\begin{tikzcd}[ampersand replacement=\&,cramped]
	0 \& {\alpha_p} \& \Ghat_a \& \Ghat_a \& 0 \\
	0 \& {\alpha_p} \& \Ghat_a \& \Ghat_a \& 0
	\arrow[from=1-1, to=1-2]
	\arrow[from=1-2, to=1-3]
	\arrow["c"', from=1-2, to=2-2]
	\arrow["F", from=1-3, to=1-4]
	\arrow["c"', from=1-3, to=2-3]
	\arrow[from=1-4, to=1-5]
	\arrow["{c^p}"', from=1-4, to=2-4]
	\arrow[from=2-1, to=2-2]
	\arrow[from=2-2, to=2-3]
	\arrow["F", from=2-3, to=2-4]
	\arrow[from=2-4, to=2-5]
\end{tikzcd}.
\]
Hence $M^{\Oort}$ is not fully faithful if $R$ is not reduced.

The argument in the proof of \cite[Lemma 4.1]{Oor74} relies on the assertion that $\Hom(\What,G) = 0$ when $G \in C(R)$,  which does not hold for general $R$. For example, when $R = \Fp[\lambda]/(\lambda^p)$, there is a nonzero map $\What \to \alpha_p$ given by $(x_0,x_1,\ldots) \mapsto \lambda x_0$.     
    Thus the exactness claimed in \cite[Lemma 4.1]{Oor74} fails, and the proof of full faithfulness in \cite[Theorem 4.2]{Oor74} fails similarly.
\end{example}

\subsection{Essential image of the Cartier--Dieudonné functor}

The goal of this subsection is to identify the essential image of the functor $\uM: C(R) \to D(\E_R)$. By Proposition \ref{prop: resolution of MG}, every such object is equivalent to a two-term complex of Cartier modules of formal Lie groups. Thus, we begin with a criterion for identifying such complexes. 

\begin{definition}
If $X \in D(\E_R)$ and $n \geq 1$, the derived quotient $X/^{\bb{L}}V^nX$ is defined to be
\be
X/^{\bb{L}}V^nX := (\E_R/V^n\E_R)\Lotimes{\E_R} X \in D(\Ab).
\ee
$X$ is \emph{derived $V$-complete} if the natural map $X \to R\ilim_n X/^{\bb{L}}V^nX$ is an equivalence in $D(\Ab)$.
\end{definition}

By construction, a finite complex $X$ of Cartier modules of formal Lie groups over $R$ is derived $V$-complete and $X/^{\bb{L}}VX  \in \Perf(R)$. We will show these conditions are also sufficient in Theorem \ref{theorem: cartier approximation}.

Note that if $M$ is an $\E_R$-module, then $M/V^n M$ has a $W_n(R)$-module structure, as $W(R)V \subseteq VW(R)$ and $v^n(W(R)) \subseteq V^n\E_R$ in $\E_R$.
The operator $F$ is not necessarily defined on $M/V^nM$, as $FV^n$ is an element of $V^{n-1}\E_R$ but not necessarily $V^n \E_R$.
However, the condition $FV^nM \subseteq V^{n-1}M$ allows us to define the action of $F$ on the $V$-adic completion $\ilim_n M/V^nM$. 
Moreover, if $x \in \E_R$, then $x$ also acts on the $V$-adic completion as follows: given $n$, write 
\[x = \sum_{i<n, j} V^i[c_{ij}]F^j + V^ny\]
for some $y \in \E_R$. If $d = \max \{j \mid \text{there is $i < n$ such that $c_{ij}\neq 0$}\}$, then $xV^{d+n} \in V^n \E_R$. Thus multiplication by $x$ descends to $M/V^{d+n}M \to M/V^nM$, which in the limit defines multiplication by $x$ on $\ilim_n M/V^nM$. This defines an $\E_R$-module structure on $\ilim_n M/V^nM$.
The natural map $M \to \ilim_n M/V^nM$ is then a homomorphism of $\E_R$-modules.

Now suppose that $X \in D(\E_R)^{\leq 0}$. Represent $X$ by a complex of free $\E_R$-modules $X^\bullet$ concentrated in degrees $\leq 0$.
Then the termwise completion $X^{\bullet\wedge}$ has underlying abelian group $R\ilim_n X^\bullet/^{\mathbb{L}} V^n X^\bullet$, and $X$ is derived $V$-complete if and only if the natural map $X^\bullet \to X^{\bullet\wedge}$ is a quasi-isomorphism. 
The object $X^\wedge$ represented by $X^{\bullet\wedge}$ in $D(\E_R)$ does not depend on the choice of complex $X^\bullet$ representing $X$.
Note that $X/^{\bb{L}}V^nX$ is not an $\E_R$-module in general, but $X^\wedge$ is.

\begin{lemma}\label{lemma: one-step cartier approximation}
    Suppose that $X \in D(\E_R)^{\leq 0}$ is derived $V$-complete.
    Let $x_1,\ldots, x_r$ be generators for $H^0(X/^{\bb{L}}VX)$ as an $R$-module.
    Then there exists a Cartier module of a formal Lie group $C$ where $C/VC$ is a free $R$-module of rank $r$ with basis $e_1,\ldots, e_r$
    and a map $C \to X$ in $D(\E_R)$ such that the induced map $C/VC = H^0(C/^{\bb{L}}VC) \to H^0(X/^{\bb{L}}VX)$ sends $e_i \mapsto x_i$.
\end{lemma}
\begin{proof}
    Represent $X$ by a complex of free $\E_R$-modules $X^\bullet$ concentrated in degrees $\leq 0$. Since $X$ is derived $V$-complete, $X^\bullet \to X^{\bullet\wedge}$ is a quasi-isomorphism. We will construct a commutative diagram
    \[
\begin{tikzcd}[ampersand replacement=\&,cramped]
	{\E_R^r} \& {X^{-1,\wedge}} \\
	{\E_R^r} \& {X^{0,\wedge}}
	\arrow["f"', from=1-1, to=1-2]
	\arrow["\varphi"', from=1-1, to=2-1]
	\arrow["d", from=1-2, to=2-2]
	\arrow["g"', from=2-1, to=2-2]
\end{tikzcd}
\]
where, if $e_1,\ldots, e_r$ is the standard basis for $\E_R^r$, $\varphi$ is of the form
\[ \varphi(e_i) = Fe_i - \sum_{n,j} V^n[c_{i,n,j}]e_j.\]
    Then $\varphi$ is a structural equation, so by Proposition \ref{prop: structure equations}, $\varphi$ is injective and $C := \coker(\varphi)$ is a Cartier module of a formal Lie group. This gives a map $C \to X^\wedge$; composing with the inverse of the quasi-isomorphism $X \to X^\wedge$ in $D(\E_R)$ gives the desired map $C \to X$.

Begin by picking $g(e_i)$ to be any lift of $x_i \in H^0(X/^{\bb{L}}VX) = X^0/(dX^{-1} + VX^0)$ to $X^0$.
We will inductively construct a sequence of $\E_R$-linear maps $\varphi_n : \E_R^r \to \E_R^r$ and $f_n: \E_R^r \to X^{-1, \wedge}$ such that $\varphi_n$ is a structural equation, 
$\varphi_n(e_i)$ and $f_n(e_i)$ converge in the $V$-adic topology for all $i$ and 
\[ df_n(e_i) - g\varphi_n(e_i) \in V^{n}X^{0,\wedge}.\]
Then defining $\varphi(e_i) = \ilim_n \varphi_n(e_i)$ and $f(e_i) = \ilim_n f_n(e_i)$ gives the desired construction.

For the base case, note that we may write the image of $Fg(e_i)$ in $H^0(X/^{\bb{L}}VX)$ as $\sum_j c_{i,0,j} x_j$ for some $c_{i,0,j} \in R$.
Define $\varphi_1(e_i) = Fe_i - \sum_j [c_{i,0,j}]e_j$;
then $g\varphi_1(e_i) \in VX^{0,\wedge} + dX^{-1,\wedge}$.
Pick $f_1(e_i)$ to be any element of $X^{-1,\wedge}$ such that 
\[ df_1(e_i) - g\varphi_1(e_i) \in V X^{0,\wedge}.\]
Now suppose that we are given $f_n$ and $\varphi_n$ such that $df_n(e_i) - g\varphi_n(e_i) \in V^nX^{0,\wedge}$.
Since $X^\bullet$ is a free $\E_R$-module in each degree, $V^n: X^\bullet/VX^\bullet \to V^nX^\bullet/V^{n+1}X^\bullet$ is an isomorphism of complexes of abelian groups. Hence there exists $c_{i,n,j} \in R$ such that 
\[ 
    df_n(e_i) - g\varphi_n(e_i) \equiv -\sum_j V^n [c_{i,n,j}]g(e_j) \mod V^{n+1}X^{0,\wedge} + dV^nX^{-1,\wedge}.
\]
Set $\varphi_{n+1}(e_i) = \varphi_n(e_i) - \sum_j V^n[c_{i,n,j}]e_j$.
Then there exist $V^n \alpha_i \in V^n X^{-1,\wedge}$ such that 
\[ df_n(e_i) - g\varphi_{n+1}(e_i) \equiv dV^n \alpha_i \mod V^{n+1}X^{0,\wedge}.\]
Define 
$f_{n+1}(e_i) = f_n(e_i) - V^n\alpha_i$. This gives the required construction.
\end{proof}

\begin{theorem}\label{theorem: cartier approximation}
    Suppose $X \in D(\E_R)$ satisfies
    \begin{enumerate}
        \item $X$ is derived $V$-complete;
        \item $X/^{\mathbb L}VX$ is a perfect complex of $R$-modules of tor-amplitude $[a,b]$.
    \end{enumerate}
    Then $X$ is equivalent to a complex of Cartier modules of formal Lie groups in degrees $[a,b]$.
\end{theorem}
\begin{proof}
    Without loss of generality, $b=0$. We induct on $a\leq 0$.
    If $a = 0$, then $X/^{\mathbb L}VX = \E_R/V\E_R \Lotimes{\E_R} X$  is a finitely generated projective $R$-module \cite[\href{https://stacks.math.columbia.edu/tag/0658}{Tag 0658}]{stacks-project}.
    Now for each $n$ there is an exact triangle
    \begin{equation}\label{eq: v-adic filtration}
        V^n\E_R/V^{n+1}\E_R \to \E_R/V^{n+1}\E_R \to \E_R/V^n\E_R \overset{+1}{\to}
    \end{equation}
    of right $\E_R$-modules,
    and $V^n: \E_R/V\E_R \to V^n\E_R/V^{n+1}\E_R$ is an isomorphism of right $\E_R$-modules.
    Tensoring this triangle with $X$ and inducting on $n$ shows
    $X/^{\mathbb L}V^nX$ is concentrated in degree 0 for all $n$ and the inverse system $\{X/^{\bb{L}}V^nX\}_{n \geq 0}$ has surjective transition maps. Hence $X$ is concentrated in degree 0.
    Now tensor the triangle $V\E_R \to \E_R \to \E_R/V\E_R \overset{+1}{\to}$ with $X$.
    Since $X/^{\mathbb L}VX$ is concentrated in degree $0$, we conclude $V: X \to X$ is injective.
    Since the transition maps in the system $\{X/^{\bb{L}}V^nX\}_{n \geq 0}$ are surjective, the derived limit $R\ilim_n X/^{\mathbb L}V^nX$ is the (non-derived) limit $\ilim_n X/V^nX$; since $X$ is derived $V$-complete, $X \to \ilim_n X/V^nX$ is an isomorphism. Thus $X$ is $V$-reduced; since $X/VX$ is a finitely generated projective $R$-module, Cartier theory (Theorem \ref{theorem: cartier in countable type}) implies that $X$ is the Cartier module of a formal Lie group.
    
    Now suppose $a < 0$. As before, tensoring with \eqref{eq: v-adic filtration} shows that $X \in D(\E_R)^{\leq 0}$. By Lemma \ref{lemma: one-step cartier approximation}, there exists a Cartier module of a formal Lie group $C$ and a map $C \to X$ such that $H^0(C/^{\bb{L}}VC) \to H^0(X/^{\bb{L}}VX)$ is surjective. Let $Y = \Fib(C \to X)$. Since $C$ and $X$ are derived $V$-complete, so is $Y$. Further, $Y/^{\bb{L}}VY = \Fib(C/^{\bb{L}}VC \to X/^{\bb{L}}VX)$ is a perfect complex of $R$-modules; since $H^0(C/^{\bb{L}}VC) \to H^0(X/^{\bb{L}}VX)$ is surjective, $Y/^{\bb{L}}VY$ has tor-amplitude in $[a+1,0]$.
    By induction, $Y$ is equivalent to a complex of Cartier modules of formal Lie groups in degrees $[a+1,0]$; then $X = \Cofib(Y\to C)$ is equivalent to a complex of Cartier modules of formal Lie groups in degrees $[a,0]$.
\end{proof} 

\begin{theorem}[Essential image] \label{theorem: essentially surjective}
    Let $R$ be a $p$-nilpotent ring.
    Then $M$ is in the essential image of $\uM: C(R) \to D(\E_R)$ if and only if
    \begin{enumerate}
        \item $M$ is derived $V$-complete;
        \item $M/^{\bb{L}}VM$ is a perfect complex of $R$-modules of tor-amplitude $[0,1]$;
        \item $L\Div\left(\E_{\ol{R}} \Lotimes{\E_R} M\right) = 0$.
    \end{enumerate}
\end{theorem}
\begin{proof}
    Suppose $M = \uM(G)$ for $G \in C(R)$. If $0 \to G \to K^0 \to K^1 \to 0$ is a resolution of $G$ by formal Lie groups, then by Proposition \ref{prop: resolution of MG}, $\uM(G) \simeq [M(K^0) \to M(K^1)]$ is equivalent to a two-term complex of Cartier modules, so the first two properties hold. Since $K^0 \to K^1$ is an isogeny, Proposition \ref{proposition: derived isogenies} implies $L\Div(\E_{\ol{R}} \Lotimes{\E_R} M) = 0$.

    Conversely, suppose $M$ is derived $V$-complete and $M/^{\bb{L}}VM$ is a perfect complex of $R$-modules of tor-amplitude $[0,1]$. 
    By Theorem \ref{theorem: cartier approximation}, $M \simeq [M(K^0) \to M(K^1)]$ where $K^0$ and $K^1$ are formal Lie groups over $R$. Let $f: K^0 \to K^1$ be the corresponding map.
    Since $L\Div(\E_{\ol{R}} \Lotimes{\E_R} M) = 0$, Proposition \ref{proposition: derived isogenies} shows $f$ is an isogeny.  Thus there is a short exact sequence
    \[ 0 \to G \to K^0 \to K^1 \to 0\]
    in the fpqc topology where $G \in C(R)$, so Proposition \ref{prop: resolution of MG} implies $\uM(G) \simeq [M(K^0) \to M(K^1)] \simeq M$.
\end{proof}

Combining Theorems \ref{theorem: fully faithful} and \ref{theorem: essentially surjective} completes the proof of Theorem \ref{theorem: main}.

\section{Properties of the Cartier--Dieudonné functor}

\subsection{The $V$-adic Filtration on $\uM(G)$} 
Fix $G \in C(R)$ and set $M := \uM(G)$.  In this subsection, we study the quotients $M_n := M/^{\bb{L}}V^nM$ in the derived category of $W(R)$-modules.  By tensoring the exact triangle
\be
V^n\E_R/V^{n+1}\E_R \to \E_R/V^{n+1}\E_R \to \E_R/V^n\E_R \overset{+1}\to 
\ee
of right $\E_R$-modules with $M$ we obtain an exact triangle
\be
(M_1)|_{\sigma^n} \overset{V^n}\to M_{n+1} \to M_n \overset{+1}\to
\ee
in $D(W(R))$, where $|_{\sigma^n}$ is the restriction of scalars by the $n$-th Frobenius $\sigma^n: W(R) \to W(R)$. Since $M_1$ has tor-amplitude in $[0, 1]$, we find by induction and the long exact sequence in cohomology that $H^i(M_n) = 0$ for $i \notin \{0, 1\}$.

Via Cartier duality \cite[Proposition 4.4.1]{AM25}, $M \cong \tau^{\leq 1} R\Hom(G^\vee, W)$, where $W$ is the ring scheme of $p$-typical Witt vectors. The short exact sequence 
\be
0 \to W|_{\sigma^n} \overset{V^n}\to W \to W_n \to 0
\ee
yields a fiber sequence
\be
R\Hom(G^\vee, W)|_{\sigma^n} \overset{V^n}\to R\Hom(G^\vee, W) \to R\Hom(G^\vee, W_n) \overset{+1}\to.
\ee
Therefore the canonical map from the cone of a truncation to the truncation of a cone produces a canonical map
\be
f_n: M_n \to \tau^{\leq 1} R\Hom(G^\vee, W_n)
\ee
which induces an isomorphism $H^0(M_n) \simeq \Hom(G^\vee, W_n)$ and an injection $H^1(M_n) \hookrightarrow \Ext^1(G^\vee, W_n)$.

\begin{proposition} \label{proposition: truncation RHom}
The canonical map $f_n: M_n \to \tau^{\leq 1} R\Hom(G^\vee, W_n)$ is an isomorphism in $D(W(R))$.
\end{proposition}

Since $f_n$ is a map between objects with cohomology only in degrees 0 and 1, and it induces an isomorphism on $H^0$, to prove the proposition it suffices to prove that the injection \be
H^1(f_n): H^1(M_n) \hookrightarrow \Ext^1(G^\vee, W_n)
\ee
is an isomorphism. We begin with the following lemma.

\begin{lemma} \label{lemma: Ext to formal Lie is zero}
Let $0 \to G \to K \to L \to 0$ be a resolution of $G$ by commutative formal Lie groups over $R$.  Then the induced map $\Ext^1(G^\vee, \G_a) \to \Ext^1(K^\vee, \G_a)$ is zero.
\end{lemma}

\begin{proof}
We have a distinguished triangle
\be
\ell_G^\vee \to \ell_K^\vee \to \ell_L^\vee \overset{+1}\to
\ee
in $D(R)$. Since $K$ is formally smooth, we have $\ell_K^\vee = \Lie(K)$ and similarly for $L$. Therefore
\be
\ell_G^\vee \cong \Fib(\Lie(K) \to \Lie(L)).
\ee
Let $f: \Lie(L) \to H^1(\ell_G^\vee)$ be the induced map. Grothendieck's formula for the dual of the co-Lie complex \cite[§14]{MM74} gives identifications $\Lie(L) \cong \Hom(L^\vee, \G_a)$ and $H^1(\ell_G^\vee) \cong \Ext^1(G^\vee, \G_a)$. Under these identifications, $f$ is identified with the boundary map $\Hom(L^\vee, \G_a) \overset{\delta}\to \Ext^1(G^\vee, \G_a)$.  It follows that $\delta$ is surjective, and therefore $\Ext^1(G^\vee, \G_a) \to \Ext^1(K^\vee, \G_a)$ is zero.
\end{proof}

\begin{proof}[Proof of Proposition \ref{proposition: truncation RHom}]
We begin by showing that $f_1$ is an isomorphism.  On $H^1$, the map is the injection
\be
\coker(V: \Ext^1(G^\vee, W)|_{\sigma} \to \Ext^1(G^\vee, W)) \hookrightarrow \Ext^1(G^\vee, \G_a),
\ee
so it suffices to show that the map $\Ext^1(G^\vee, W) \to \Ext^1(G^\vee, \G_a)$ induced by the surjection $W \to \G_a$ is surjective.
Let $0 \to G \to K \to L \to 0$ be a resolution by commutative formal Lie groups.  We have a commutative square
\be
\begin{tikzcd}
\Ext^1(G^\vee, \G_a) \arrow[d, "0"] \arrow[r] & \Ext^2(G^\vee, W) \arrow[d] \\
\Ext^1(K^\vee, \G_a) \arrow[r] & \Ext^2(K^\vee, W)
\end{tikzcd}
\ee
where the left vertical map is zero by Lemma \ref{lemma: Ext to formal Lie is zero}.  Moreover, the map $\Ext^2(G^\vee, W) \to \Ext^2(K^\vee, W)$ is injective: its kernel is the image of $\Ext^1(L^\vee, W)$, and we have $\Ext^1(L^\vee, W) = \Ext^1(\What, L) = 0$ by Proposition \ref{prop: ext1 from What}.  It follows that the map $\Ext^1(G^\vee, \G_a) \to \Ext^2(G^\vee, W)$ is zero, and consequently the map $\Ext^1(G^\vee, W) \to \Ext^1(G^\vee, \G_a)$ is surjective, as desired.

For $f_n$, we have an exact sequence
\be
\Ext^1(G^\vee, W)|_{\sigma^n} \overset{V^n}\to \Ext^1(G^\vee, W) \to \Ext^1(G^\vee, W_n) \to \Ext^2(G^\vee, W)|_{\sigma^n} \overset{V^n}\to \Ext^2(G^\vee, W)
\ee
yielding a short exact sequence
\be
0 \to H^1(M_n) \to \Ext^1(G^\vee, W_n) \to \ker(V^n: \Ext^2(G^\vee, W)|_{\sigma^n} \to \Ext^2(G^\vee, W)) \to 0.
\ee
By the $n = 1$ case, the map $H^1(M_1) \to \Ext^1(G^\vee, \G_a)$ is an isomorphism, so $V: \Ext^2(G^\vee, W)|_{\sigma} \to \Ext^2(G^\vee, W)$ is injective.  Since $V^n: \Ext^2(G^\vee, W)|_{\sigma^n} \to \Ext^2(G^\vee, W)$ is the composition $\Ext^2(G^\vee, W)|_{\sigma^n} \overset{V}\to \Ext^2(G^\vee, W)|_{\sigma^{n-1}} \overset{V}\to \dots \to \Ext^2(G^\vee, W)$ and restriction of scalars preserves injectivity, we have that $V^n: \Ext^2(G^\vee, W)|_{\sigma^n} \to \Ext^2(G^\vee, W)$ is injective, so $H^1(M_n) \to \Ext^1(G^\vee, W_n)$ is an isomorphism, and we are done.
\end{proof}

This leads to the following result on the $V$-adic structure of $M$.

\begin{proposition} \label{proposition: extension of Lie}
For each $n \geq 1$, $M_n \in D(W(R))$ is an iterated extension of the Witt Frobenius restricted Lie complexes $\ell_G^\vee, \ell_G^\vee|_{\sigma}, \dots, \ell_G^\vee|_{\sigma^{n-1}}$.
\end{proposition}

\begin{proof}
By Grothendieck's formula \cite[§14]{MM74}, we have a canonical isomorphism $\ell_G^\vee \cong \tau^{\leq 1}R\Hom(G^\vee, \G_a)$. Therefore the result follows from Proposition \ref{proposition: truncation RHom} and induction applied to the triangles
\[
(M_1)|_{\sigma^n} \overset{V^n}\to M_{n+1} \to M_n \overset{+1}\to.
\qedhere
\]
\end{proof}

\subsection{Base change}

\begin{proposition} \label{proposition: base change}
    Suppose that $G \in C(R)$ and $R \to R'$ is a morphism of $p$-nilpotent rings. Let $G' = G \times_{\Spec R} \Spec R'$. Then 
    \[ \uM(G') \cong \E_{R'} \Lotimes{\E_{R}} \uM(G).\]
\end{proposition}
\begin{proof}
Let $0 \to G \to K \to L \to 0$ be a resolution of $G$ by formal Lie groups, and let $0 \to G' \to K' \to L' \to 0$ be the base change of this resolution to $\Spec R'$.  By Proposition \ref{prop: resolution of MG}, we have $\uM(G) \cong \Fib(M(K) \to M(L))$ and $\uM(G') \cong \Fib(M(K') \to M(L'))$.  Applying derived base change to $\E_{R'}$ yields a distinguished triangle
\be
\E_{R'} \Lotimes{\E_{R}} \uM(G) \to \E_{R'} \Lotimes{\E_R} M(K) \to \E_{R'} \Lotimes{\E_R} M(L) \overset{+1}\to
\ee
in $D(\E_{R'})$. By Lemma \ref{lemma: tor vanishing}, we have $\E_{R'} \Lotimes{\E_R} M(K) \cong \E_{R'} \ten_{\E_R} M(K)$, so by compatibility of Cartier theory with base change \cite[Theorems 4.43 and 4.45]{Zin84} we have $\E_{R'} \ten_{\E_R} M(K) \cong M(K')$.  As the same reasoning applies to $L$, we deduce
\be
\E_{R'} \Lotimes{\E_{R}} \uM(G) \cong \Fib(M(K') \to M(L')) \cong \uM(G')
\ee
as desired.
\end{proof}

\subsection{The case of a perfect field}

Suppose that $k$ is a perfect field of characteristic $p$.
If $G \in C(k)$, then the Cartier dual $G^\vee$ of $G$ is unipotent, and the classical Dieudonné module associated to $G^\vee$ is 
\[ D(G^\vee) = \colim_m \Hom(G^\vee, W_m)|_{\sigma^{-m}}.\]
Suppose that $F^n_G = 0$. Then $D(G^\vee) = \Hom(G^\vee, W_n)|_{\sigma^{-n}}$.  
To relate $D(G^\vee)$ to $M(G)$, we consider the long exact sequence formed by applying $R\Hom(G^\vee,-)$ to 
\[ 0 \to W^{(p^n)} \overset{V^n}\to W \to W_n \to 0.\]

\begin{lemma}\label{lemma: perfect base no maps to W}
    If $R$ is a perfect $\Fp$-algebra and $G \in C(R)$, then $\Hom(G^\vee, W) = 0$.
\end{lemma}
\begin{proof}
    Pick $n$ such that $F^n_G = 0$.
    If $f : G^\vee \to W$, then the following diagram commutes:
    \[
\begin{tikzcd}[cramped]
	{(G^\vee)^{(p^n)}} & {G^\vee} \\
	{W^{(p^n)}} & W
	\arrow["{V^n_{G^\vee}}", from=1-1, to=1-2]
	\arrow["{f^{(p^n)}}"', from=1-1, to=2-1]
	\arrow["f"', from=1-2, to=2-2]
	\arrow["{V^n_W}"', from=2-1, to=2-2]
\end{tikzcd}
    \]
    Since $V^n_W$ is injective and $V^n_{G^\vee} = 0$, we conclude $f^{(p^n)} = 0$. Since $R$ is perfect, $f=0$.
\end{proof}

\begin{lemma}\label{lemma: naturality of V on Ext1}
    Let $R$ be an $\Fp$-algebra and let $A$ and $B$ be flat affine commutative group schemes over $R$.
    Then the following diagram commutes:
    \[
\begin{tikzcd}[cramped]
	{\Ext^1(A,B^{(p)})} & {\Ext^1(A,B^{(p)})} \\
	{\Ext^1(A,B)} & {\Ext^1(A^{(p)},B^{(p)})}
	\arrow["{=}", from=1-1, to=1-2]
	\arrow["{(V_B)_*}", from=1-1, to=2-1]
	\arrow["{(V_A)^*}"', from=1-2, to=2-2]
	\arrow[from=2-1, to=2-2]
\end{tikzcd}
    \]
    where the arrow $\Ext^1(A,B) \to \Ext^1(A^{(p)},B^{(p)})$ is induced by the Frobenius twist, and $\Ext$ is taken in the fpqc topology over $\Spec(R)$.
\end{lemma}
\begin{proof}
    Suppose $0 \to B^{(p)} \to E \to A \to 0$ is an fpqc extension of sheaves. By descent, $E$ is also a flat affine group scheme over $R$ (e.g.\ by \cite[Proposition 17.4]{Oor66}\footnote{In \cite{Oor66}, Oort assumes the base is locally Noetherian, but this is not used in the proof of Proposition 17.4.}).
    Let $E'$ be the pushout of $E$ along $V_B: B^{(p)} \to B$.
    We wish to show $(E')^{(p)} \cong E \times_{A, V_A} A^{(p)}$ as an extension of $A^{(p)}$ by $B^{(p)}$.
    Consider the following diagram:
\[\begin{tikzcd}
	0 & {B^{(p^2)}} & {E^{(p)}} & {A^{(p)}} & 0 \\
	0 & {B^{(p)}} & {(E')^{(p)}} & {A^{(p)}} & 0 \\
	0 & {B^{(p)}} & {E \times_{A,V}A^{(p)}} & {A^{(p)}} & 0 \\
	0 & {B^{(p)}} & E & A & 0
	\arrow[from=1-1, to=1-2]
	\arrow[from=1-2, to=1-3]
	\arrow["V", from=1-2, to=2-2]
	\arrow[from=1-3, to=1-4]
	\arrow[from=1-3, to=2-3]
	\arrow["{V_E}"{description}, curve={height=24pt}, dashed, from=1-3, to=4-3]
	\arrow[from=1-4, to=1-5]
	\arrow["{=}"{description}, from=1-4, to=2-4]
	\arrow[from=2-1, to=2-2]
	\arrow[from=2-2, to=2-3]
	\arrow["{=}"{description}, from=2-2, to=3-2]
	\arrow["\lrcorner"{anchor=center, pos=0.125, rotate=180}, draw=none, from=2-3, to=1-2]
	\arrow[from=2-3, to=2-4]
	\arrow[from=2-4, to=2-5]
	\arrow["{=}"{description}, from=2-4, to=3-4]
	\arrow[from=3-1, to=3-2]
	\arrow[from=3-2, to=3-3]
	\arrow[from=3-2, to=4-2]
	\arrow[from=3-3, to=3-4]
	\arrow[from=3-3, to=4-3]
	\arrow["\lrcorner"{anchor=center, pos=0.125}, draw=none, from=3-3, to=4-4]
	\arrow[from=3-4, to=3-5]
	\arrow["V", from=3-4, to=4-4]
	\arrow[from=4-1, to=4-2]
	\arrow[from=4-2, to=4-3]
	\arrow[from=4-3, to=4-4]
	\arrow[from=4-4, to=4-5]
\end{tikzcd}\]
    Since $E$ is a flat affine group scheme, $V_E$ exists and makes the long vertical squares in the diagram commute.
By definition of pushout and pullback, $V_E$ induces a map $(E')^{(p)} \to E \times_{A,V_A} A^{(p)}$ which is compatible with the extension structure, as desired.
\end{proof}

\begin{corollary} \label{corollary: perfect field comparison}
    Let $k$ be a perfect field of characteristic $p>0$. If $G \in C(k)$, then 
    \[ \uM(G) \cong D(G^\vee)[-1].\]
\end{corollary}
\begin{proof}
    Suppose $F^n_G = 0$. Then $D(G^\vee) = \Hom(G^\vee, W_n)|_{\sigma^{-n}} = \Hom(W_n^\vee, G)|_{\sigma^{-n}}$.
    By Lemma \ref{lemma: perfect base no maps to W}, $\Hom(G^\vee, W) = 0$, so $\uM(G) \cong \Ext^1(\What,G)[-1]$. By \cite[Proposition 4.4.1]{AM25}, Cartier duality induces an equivalence
    \[ \Ext^1(\What,G) \cong \Ext^1(G^\vee, W).
    \]
    By applying $R\Hom(G^\vee,-)$ to 
    \[ 0 \to W \overset{V^n}{\to} W|_{\sigma^{-n}} \to W_n|_{\sigma^{-n}} \to 0, 
    \]
    it suffices to show $(V^n_W)_*: \Ext^1(G^\vee,W) \to \Ext^1(G^\vee,W)$ is zero.
    By Lemma \ref{lemma: naturality of V on Ext1}, $(V^n_W)_*$ agrees with $(V^n_{G^\vee})^* = 0$ up to a twist. Since $k$ is perfect, $(V^n_W)_*: \Ext^1(G^\vee,W) \to \Ext^1(G^\vee,W)|_{\sigma^{-n}}$ is zero, which proves the claim.
\end{proof}

\begin{remark}
    The same argument shows that if $R$ is a perfect $\Fp$-algebra and $G \in C(R)$, then 
    \[ 
    \uM(G) \cong \Ext^1_{\fpqc}(G^\vee,W)[-1].\]
\end{remark}

\subsection{Exactness and extensions}

\begin{proposition}\label{prop: uM is exact}
    Let $R$ be a $p$-nilpotent ring and $G' \to G \to G''$ be a sequence of morphisms in $C(R)$.
    Then
    \[ 0 \to G' \to G \to G'' \to 0\]
    is exact in the fpqc topology
    if and only if 
    \[ \uM(G') \to \uM(G) \to \uM(G'')\]
    is a fiber sequence in $D(\E_R)$.
\end{proposition}
\begin{proof}
    First suppose that $0 \to G' \to G \to G'' \to 0$ is exact in the fpqc topology.
    By definition, $\uM = \tau^{\leq 1} R\Hom_{\fpqc}(\What,-)$, so it suffices to show that 
    \[
        \Ext^1_{\fpqc}(\What, G) \to \Ext^1_{\fpqc}(\What,G'')
    \]
    is surjective.
    Let $0 \to G \to K^0 \to K^1 \to 0$ be a short exact sequence in the fpqc topology where $K^0$ and $K^1$ are formal Lie groups.
    By Theorem \ref{theorem: quotients of FLG}, $K^0/G'$ is also a formal Lie group, and by the third isomorphism theorem, 
    \[ 0 \to G'' \to K^0/G' \to K^1 \to 0\]
    is exact. Thus we have a commutative diagram 
    \[
\begin{tikzcd}[ampersand replacement=\&]
	0 \& G \& {K^0} \& {K^1} \& 0 \\
	0 \& {G''} \& {K^0/G'} \& {K^1} \& 0
	\arrow[from=1-1, to=1-2]
	\arrow[from=1-2, to=1-3]
	\arrow[from=1-2, to=2-2]
	\arrow[from=1-3, to=1-4]
	\arrow[from=1-3, to=2-3]
	\arrow[from=1-4, to=1-5]
	\arrow["{=}"{description}, from=1-4, to=2-4]
	\arrow[from=2-1, to=2-2]
	\arrow[from=2-2, to=2-3]
	\arrow[from=2-3, to=2-4]
	\arrow[from=2-4, to=2-5]
\end{tikzcd}
\]
with exact rows.
Applying $R\Hom_{\fpqc}(\What,-)$ gives the diagram 
\[
\begin{tikzcd}[cramped]
	{\Hom(\What,K^1)} & {\Ext^1_{\fpqc}(\What,G)} & {\Ext^1_{\fpqc}(\What,K^0)} \\
	{\Hom(\What,K^1)} & {\Ext^1_{\fpqc}(\What,G'')} & {\Ext^1_{\fpqc}(\What,K^0/G')}
	\arrow[from=1-1, to=1-2]
	\arrow["{=}"{description}, from=1-1, to=2-1]
	\arrow[from=1-2, to=1-3]
	\arrow[from=1-2, to=2-2]
	\arrow[from=1-3, to=2-3]
	\arrow[from=2-1, to=2-2]
	\arrow[from=2-2, to=2-3]
\end{tikzcd}
\]
with exact rows.
By Proposition \ref{prop: ext1 from What}, $\Ext^1_{\fpqc}(\What,K^0) = 0$ and $\Ext^1_{\fpqc}(\What, K^0/G')=  0$,
so $\Ext^1_{\fpqc}(\What,G) \to \Ext^1_{\fpqc}(\What,G'')$ is surjective.

Now suppose that $\uM(G') \to \uM(G) \to \uM(G'')$ is a fiber sequence in $D(\E_R)$. Since $\uM$ is fully faithful, the composition $G' \to G \to G''$ is zero, so using \cite[Proposition 1.1]{dJon93} we may check the exactness of $0 \to G' \to G \to G'' \to 0$ on geometric fibers. Let $s: \Spec k \to \Spec R$ be a geometric point. By Proposition \ref{proposition: base change}, $\uM$ is compatible with base change, and therefore applying derived base change to $\E_{k}$ yields a fiber sequence
\be
M(G'_{s}) \to M(G_{s}) \to M(G''_{s}) \overset{+1}\to
\ee
in $D(\E_{k})$. By Corollary \ref{corollary: perfect field comparison}, the long exact sequence associated to this fiber sequence is
\be
0 \to D((G'_s)^\vee) \to D(G_s^\vee) \to D((G_s'')^\vee) \to 0
\ee
where $D(-)$ is the classical contravariant Dieudonn\'e module functor over $k$.  Since classical Dieudonn\'e theory is an equivalence of abelian categories, it is exact, so $0 \to G'_{s} \to G_{s} \to G''_{s} \to 0$ is exact, as desired. 
\end{proof}

\begin{lemma}\label{lemma: C(R) closed under extensions}
    Let $R$ be a $p$-nilpotent ring. 
    If $G',G'' \in C(R)$ and 
    \[ 0 \to G' \to G \to G'' \to 0\]
    is a short exact sequence of abelian sheaves in the fpqc topology,
    then $G \in C(R)$.
\end{lemma}
\begin{proof}
    By \cite[Proposition 17.4]{Oor66}, $G$ is representable by an affine group scheme over $R$. 
    Since $G' \to \Spec R$ is finite locally free, fpqc descent \cite[\href{https://stacks.math.columbia.edu/tag/02VO}{Tag 02VO}]{stacks-project} implies $G \to G''$ is a finite locally free morphism, so $G \to \Spec R$ is finite locally free.

    We must show $G$ is fiberwise connected.
    By Proposition \ref{proposition: connected conditions}, it suffices to show the augmentation ideal of the coordinate ring of $G$ is nilpotent.
    Suppose $G = \Spec A, G' = \Spec A', G'' = \Spec A''$, with augmentation ideals $I \subseteq A, I' \subseteq A', I'' \subseteq A''$. Let $\pi: G \to G''$ be the quotient map and $g \in G(A/AI'')$ correspond to $A \to A/AI''$.
    Then $\pi(g) = 0$ since $A'' \to A \to A/AI''$ factors through $A'' \to A''/I'' = R$.
    Thus $g \in \ker(\pi) = G'$, so
    the map $A \to A/AI''$ factors through $A'$.
    Since $A \onto A'$ takes $I$ onto $I'$ and $I'$ is nilpotent, there exists $n$ such that $I^n \subseteq AI''$.
    Since $I'' \subset A''$ is nilpotent,
    this implies $I$ is nilpotent.    
\end{proof}

\begin{theorem}\label{theorem: ext1 preserved}
    Let $R$ be a $p$-nilpotent ring and $G', G'' \in C(R)$.
    Then there is an isomorphism 
    \[ \Ext^1_{\fpqc}(G'',G') \cong \Hom_{D(\E_R)}(\uM(G''),\uM(G')[1])\]
    defined by sending an extension 
    \[ 0 \to G' \to G \to G'' \to 0\]
    to the cofiber $\uM(G'') \to \uM(G')[1]$
    of the induced map $\uM(G) \to \uM(G'')$.
\end{theorem}
\begin{proof}
    First, we explain why sending $0\to G' \to G \to G''\to0$ to $\uM(G'') \to \uM(G')[1]$ is well-defined.
    By Lemma \ref{lemma: C(R) closed under extensions}, if $G$ is such an extension, then $G \in C(R)$.
    By Proposition \ref{prop: uM is exact}, $\uM(G') \to \uM(G) \to \uM(G'')$ is a fiber sequence in $D(\E_R)$,
    so the cofiber of $\uM(G) \to \uM(G'')$ is indeed a map $\uM(G'') \to \uM(G')[1]$.
    An isomorphism of extensions gives equivalent cofibers.
    This defines an additive map 
    \begin{equation}\label{eq: extension map}
        \Ext^1_{\fpqc}(G'',G') \to \Hom_{D(\E_R)} (\uM(G''),\uM(G')[1]).
    \end{equation}
    We claim \eqref{eq: extension map} is injective. For suppose that $\uM(G'') \to \uM(G')[1]$ is zero.
    Then there exists a splitting $u: \uM(G'') \to \uM(G)$ of the fiber sequence $\uM(G') \to \uM(G) \to \uM(G'')$. 
    Since $\uM$ is fully faithful, $u = \uM(\varphi)$ where $\varphi$ is a splitting of $0 \to G' \to G \to G'' \to 0$. 
    
    We claim that \eqref{eq: extension map} is surjective. For let $M' = \uM(G')$ and $M'' = \uM(G'')$. Suppose that $f: M'' \to M'[1]$ is a map in $D(\E_R)$.
    Let $M$ be the fiber of $f$, so that there is a fiber sequence 
    \[ M' \to M \to M''.\]
    Since $M'$ and $M''$ are derived $V$-complete, so is $M$.
    As $M'/^{\bb{L}}VM'$ and $M''/^{\bb{L}}VM''$ are perfect complexes of $R$-modules of tor-amplitude $[0,1]$, so is $M/^{\bb{L}}VM$. 
    Additionally, $L\Div(\E_{\ol{R}} \Lotimes{\E_R} M') =0$ and $L\Div(\E_{\ol{R}} \Lotimes{\E_R} M'') = 0$, so that $L\Div(\E_{\ol{R}} \Lotimes{\E_R} M) = 0$.
    By Theorem \ref{theorem: essentially surjective}, $M \simeq \uM(G)$ for some $G \in C(R)$.
    Full faithfulness of $\uM(-)$ (Theorem \ref{theorem: fully faithful}) then implies that there are group homomorphisms $G' \to G \to G''$ inducing the maps $\uM(G') \to M \to \uM(G'')$.
    By Proposition \ref{prop: uM is exact}, $0 \to G' \to G \to G'' \to 0$ is exact. Hence $f:M'' \to M'[1]$ is in the image of $[G] \in \Ext^1_{\fpqc}(G'',G')$.
\end{proof}

\printbibliography

@misc{AM25,
  archiveprefix = {arXiv},
  author        = {Dima Arinkin and Joshua Mundinger},
  eprint        = {2512.13856},
  primaryclass  = {math.AG},
  title         = {Cartier duality via Mittag-Leffler modules},
  year          = {2025}
}

@misc{BD91_Hitchin,
  author = {Beilinson, Alexander and Drinfeld, Vladimir},
  url = {https://math.uchicago.edu/~drinfeld/langlands/QuantizationHitchin.pdf},
  title  = {Quantization of {H}itchin's integrable system and {H}ecke eigensheaves},
  year   = {1991}
}

@article{Beg80,
  author   = {Bégueri, Lucile},
  fjournal = {M\'{e}moires de la Soci\'{e}t\'{e} Math\'{e}matique de France. Nouvelle S\'{e}rie},
  issn     = {0037-9484},
  journal  = {M\'{e}m. Soc. Math. France (N.S.)},
  mrnumber = {615883},
  number   = {4},
  pages    = {124},
  title    = {Dualit\'{e} sur un corps local \`a corps r\'{e}siduel alg\'{e}briquement clos},
  year     = {1980}
}

@article{Car67Groupes,
  author   = {Cartier, Pierre},
  fjournal = {Comptes Rendus Hebdomadaires des S\'{e}ances de l'Acad\'{e}mie des Sciences. S\'{e}ries A et B},
  issn     = {0151-0509},
  journal  = {C. R. Acad. Sci. Paris S\'{e}r. A-B},
  mrnumber = {218361},
  pages    = {A49--A52},
  title    = {Groupes formels associ\'{e}s aux anneaux de {W}itt g\'{e}n\'{e}ralis\'{e}s},
  volume   = {265},
  year     = {1967}
}

@article{Car67Modules,
  author   = {Cartier, Pierre},
  fjournal = {Comptes Rendus Hebdomadaires des S\'{e}ances de l'Acad\'{e}mie des Sciences. S\'{e}ries A et B},
  issn     = {0151-0509},
  journal  = {C. R. Acad. Sci. Paris S\'{e}r. A-B},
  mrnumber = {218362},
  pages    = {A129--A132},
  title    = {Modules associ\'{e}s \`a un groupe formel commutatif. {C}ourbes typiques},
  volume   = {265},
  year     = {1967}
}

@article{dJon93,
  author   = {de Jong, A. J.},
  doi      = {10.1007/BF01232664},
  fjournal = {Inventiones Mathematicae},
  issn     = {0020-9910},
  journal  = {Invent. Math.},
  number   = {1},
  pages    = {89--137},
  title    = {Finite locally free group schemes in characteristic {$p$} and {D}ieudonn\'{e} modules},
  volume   = {114},
  year     = {1993},
    MRNUMBER = {1235021},
}

@article{Dri24,
  author   = {Drinfeld, Vladimir},
  doi      = {10.4310/pamq.2024.v20.n1.a7},
  fjournal = {Pure and Applied Mathematics Quarterly},
  issn     = {1558-8599,1558-8602},
  journal  = {Pure Appl. Math. Q.},
  mrnumber = {4734873},
  number   = {1},
  pages    = {233--305},
  title    = {A 1-dimensional formal group over the prismatization of {${\rm Spf}\,\Bbb Z_p$}},
  volume   = {20},
  year     = {2024}
}

@book{Haz78,
  author    = {Hazewinkel, Michiel},
  mrnumber  = {506881},
  publisher = {Academic Press, Inc.},
  series    = {Pure and Applied Mathematics},
  title     = {Formal groups and applications},
  volume    = {78},
  year      = {1978}
}

@book{Mes72,
  author    = {Messing, William},
  fseries   = {Lecture Notes in Mathematics},
  mrnumber  = {347836},
  pagetotal = {iii+190},
  publisher = {Springer},
  series    = {Lecture Notes in Math.},
  title     = {The crystals associated to {B}arsotti-{T}ate groups: with applications to abelian schemes},
  volume    = {264},
  year      = {1972}
}

@book{Oor66,
  author    = {Oort, Frans},
  fseries   = {Lecture Notes in Mathematics},
  mrnumber  = {213365},
  pagetotal = {vi+133},
  publisher = {Springer},
  series    = {Lecture Notes in Math.},
  title     = {Commutative group schemes},
  volume    = {15},
  year      = {1966}
}

@article{Oor74,
  author   = {Oort, Frans},
  fjournal = {Indagationes Mathematicae (Proceedings)},
  journal  = {Indag. Math.},
  mrnumber = {354694},
  number   = {3},
  pages    = {284-292},
  title    = {Dieudonn\'{e} modules of finite local group schemes},
  volume   = {77},
  year     = {1974}
}

@article{RG71,
  author   = {Raynaud, Michel and Gruson, Laurent},
  doi      = {10.1007/BF01390094},
  fjournal = {Inventiones Mathematicae},
  issn     = {0020-9910},
  journal  = {Invent. Math.},
  pages    = {1--89},
  subtitle = {{T}echniques de ``platification'' d'un module},
  title    = {Crit\`eres de platitude et de projectivit\'{e}},
  volume   = {13},
  year     = {1971},
    MRNUMBER = {308104},
}

@book{Sch94,
  author    = {Scheiderer, Claus},
  doi       = {10.1007/BFb0074269},
  fseries   = {Lecture Notes in Mathematics},
  pagetotal = {xxiv+284},
  publisher = {Springer},
  series    = {Lecture Notes in Math.},
  title     = {Real and Etale Cohomology},
  year      = {1994}
}

@book{SGA3I,
  editor     = {Demazure, Michel and Grothendieck, Alexandre},
  fseries    = {Lecture Notes in Mathematics},
  mrnumber   = {274458},
  publisher  = {Springer},
  series     = {Lecture Notes in Math.},
  shorthand  = {SGA3},
  subtitle   = {Tome I: {P}ropri\'et\'es g\'en\'erales des sch\'emas en groupes},
  title      = {Sch\'emas en groupes},
  titleaddon = {S\'eminaire de G\'eom\'etrie Alg\'ebrique du Bois Marie 1962/64 (SGA 3)},
  volume     = {151},
  year       = {1970}
}

@misc{stacks-project,
  author       = {{Stacks project authors}},
  howpublished = {\url{https://stacks.math.columbia.edu}},
  title        = {The Stacks project},
  year         = {2025}
}

@article{TO70,
  author   = {Tate, John and Oort, Frans},
  fjournal = {Annales Scientifiques de l'\'{E}cole Normale Sup\'{e}rieure. Quatri\`eme S\'{e}rie},
  issn     = {0012-9593},
  journal  = {Ann. Sci. \'{E}cole Norm. Sup. (4)},
  mrnumber = {265368},
  pages    = {1--21},
  title    = {Group schemes of prime order},
  volume   = {3},
  year     = {1970}
}

@book{Zin84,
  author    = {Zink, Thomas},
  note      = {With English, French and Russian summaries},
  pagetotal = {124},
  publisher = {BSB B. G. Teubner Verlagsgesellschaft},
  series    = {Teubner-Texte zur Mathematik [Teubner Texts in Mathematics]},
  title     = {Cartiertheorie kommutativer formaler {G}ruppen},
  volume    = {68},
  year      = {1984},
    MRNUMBER = {767090},
}

@book{MM74,
  author    = {Mazur, Barry and Messing, William},
  publisher = {Springer},
  series    = {Lecture Notes in Math.},
    fseries = {Lecture Notes in Mathematics},
  title     = {Universal extensions and one dimensional crystalline cohomology},
volume = {370},
  year      = {1974}
}

\end{document}